\documentclass[%
a4paper,twoside,leqno]{amsart}

\usepackage[OT1]{fontenc}
\usepackage[utf8]{inputenc}
\usepackage{amsmath,amssymb,amsfonts,amsthm}
\usepackage[english]{babel}
\usepackage{hyperref}

\newcommand{\e}{\mathrm{e}}
\newcommand{\R}{\mathbb{R}}
\newcommand{\A}{\mathbb{A}}
\newcommand{\B}{\mathbb{B}}

\newcommand{\N}{\mathbb{N}}
\newcommand{\C}{\mathbb{C}}
\newcommand{\Cplus}{\mathbb{C}_+}
\newcommand{\Cpluscl}{\overline{\mathbb{C}_+}}
\renewcommand{\O}{\operatorname{O}}
\newcommand{\NA}{{}^N\!\mathcal{A}}
\newcommand{\W}{\mathcal{W}}

\newcommand{\T}{\mathsf{T}}
\renewcommand{\d}{\mathrm{d}}

\newcommand{\sech}{\operatorname{sech}}
\newcommand{\spn}{\mathrm{span}\,}

\newcommand{\sign}{\mathrm{sign}\,}
\renewcommand{\Re}{\mathrm{Re}\,}
\renewcommand{\Im}{\mathrm{Im}\,}
\newcommand{\abs}[1]{\left\lvert #1 \right\rvert}

\newcommand{\scalarprod}[1]{\left\langle #1\right\rangle}

\newcommand{\KdV}{\text{\textup{KdV}}}
\newcommand{\leftup}[1]{{}^{#1}\! }
\theoremstyle{plain}
\newtheorem{thm}{Theorem}
\newtheorem*{thm*}{Theorem}
\newtheorem{lem}{Lemma}
\newtheorem{prop}{Proposition}
\newtheorem*{prop*}{Proposition}

\newtheorem*{conj*}{Conjecture}
\theoremstyle{definition}

\theoremstyle{remark}
\newtheorem*{rem}{Remark}

\numberwithin{equation}{section}

\allowdisplaybreaks

\title[Evans Function Approach]{An Evans-function approach to spectral stability of internal solitary waves in
stratified fluids}
\author{Andreas Klaiber}
\date{\today}
\thanks{\noindent\emph{Present address:} Zentrum Mathematik, Technische Universit\"at M\"unchen, Boltzmannstr.\ 3, 85747
Garching, Germany. \emph{Phone:}  +49 89 289 17900. \emph{Email:} klaiber@ma.tum.de}

\begin{document}
\begin{abstract}
Frequently encountered in nature, internal solitary waves in stratified fluids are well-observed and well-studied 
from the experimental, the theoretical, and the numerical perspective. 
From the mathematical point of view, these waves are exact solutions of the 2D Euler equations for incompressible, inviscid
fluids. Contrasting with a rich theory for their existence 
and the development of methods for computing these waves, their stability analysis has hardly received attention at a
rigorous mathematical level.

This paper proposes a new approach to the investigation of stability of internal solitary 
waves in a continuously stratified fluidic medium and carries out the following four steps of this approach:
(I)~to formulate the eigenvalue problem as an infinite-dimensional spatial-dynamical system,
(II)~to introduce finite-dimensional truncations of the spatial-dynamics description,
(III)~to demonstrate that each truncation, of any order, permits a well-defined Evans function,
(IV)~to prove absence of small zeros of the Evans function in the small-amplitude limit. 
The latter notably implies the low-frequency spectral stability of small-amplitude internal solitary waves to arbitrarily high truncation order.
\end{abstract}

\maketitle
\setcounter{section}{-1}
\section{Introduction}\label{sec:intro}
Fluidic media that are stratified according to varying density, as for example lakes, oceans, and atmospheres,
typically permit the development and propagation of so-called \textit{internal waves},
which, in contrast to the familiar surface waves, chiefly displace fluid elements far beneath the surface.
Internal waves which are close to some quiescent state both far ahead and far astern the wave are called
\textit{internal solitary waves} (ISWs). As ISWs provide important mechanisms for mixing and energy transport and thus have direct
ecological implications, the fields of ocean\-o\-gra\-phy, limnology, and atmosphere science have devoted considerable attention to their
observation and description, see \cite{Apel02,BrownChristie98,HelfrichMelville06,PreussePeetersEtAl10}.

The ``channel model'' widely used in this context is given by the 
2D Euler equations for incompressible fluids posed on a strip of constant finite height and infinite horizontal extent. 
Mathematical results based on this model broadly show the existence of solitary 
waves (see below), but the stability of ISWs within the channel model is an open problem.
There exist numerous works on comparatively simple model equations, 
e.g.\ Korteweg-deVries equation, extended Korteweg-deVries equation, and intermediate long-wave equation, 
to name just a few, for which the question of 
stability of solitary waves has been answered comprehensively.
These results certainly have implications on the stability properties of ISWs 
in the full-Euler channel-model setting. But, to the best of our knowledge, no stability analysis has
been conducted in this setting as yet.

The central object of the present paper is a sequence of $4(N\!+\!1)$-dimensional
systems, $\left(E_N\right)_{N\in\N}$, of ordinary differential equations on the real line 
that are associated with the spectral stability problem of ISWs in the channel model.
We prove two main results on these ``truncated eigenvalue problems'', 
one for essentially arbitrary ISWs, the other  for ISWs of small
amplitude. The result for arbitrary ISWs (Theorem~III in Section~3) 
establishes what is called consistent splitting on the closed right half 
$\Cpluscl$ 
of the complex plane, and thus the possibility of properly
defining an Evans function, $D_N$, on
$\Cpluscl$ 
for each $E_N$. The result for small-amplitude ISWs (Theorem~IV in Section~4) shows that each $E_N$
has no bounded solutions for small non-negative-real-part values of 
the spectral parameter, i.e., there do not exist unstable modes of small frequency. 

Before obtaining these rigorous results on the truncated problems $E_N$,
we motivate the $E_N$ from the linearization of the full Euler equations
about the ISW profile. This is done in two steps. The first step consists 
in showing that the eigenvalue problem derived from this linearization can be cast in a spatial-dynamics
formulation, $E$, on $(\mathcal{L}^2(0,1))^4$ (Theorem~I in Section~2). In the 
second step, we apply a Galerkin type procedure to $E$ and obtain 
the `truncations'  $E_N$ (Theorem~II in Section~2). We emphasize that these `derivations' of 
$E$ and the $E_N$ are completely formal; no attempts are made in 
the present paper to give the (certainly ill-posed!) `infinite-dimensional
dynamical system' $E$ a rigorous interpretation, to even only formulate
spectral stability at its level, or to show that the $E_N$ approximate
$E$ in a rigorous sense.%
\footnote{But these three questions admittedly are topics of ongoing work.}

It does seem to the author, however, that Theorems~III and~IV would be 
very strange coincidences if the $E_N$ did not, despite the formal nature 
of their deduction in Theorems~I and~II, capture essential
features of the ISW stability problem in the original full-Euler
channel-model setting. In particular, we consider 
Theorem~III as a meaningful characterization of stability properties of internal
solitary 
waves of arbitrary amplitude and Theorem~IV as significant evidence for the stability of internal solitary 
waves of small amplitude.

\medskip

In the following, we recapitulate previous work which has motivated our approach.
Kirchg\"assner proposed what is now, generally, called ``spatial dynamics'' in his study~\cite{Kirchgassner82} of an
elliptic PDE%
\footnote{He did indeed consider a variant of the Dubreil-Jacotin equation that governs ISWs, see Section~1 below.}
posed, as in our context, on a two-dimensional channel.
In this approach, the unbounded spatial variable $-\infty<x<+\infty$ is considered as the ``time'' of a dynamical system
living on some function space. Although ill-posed, this spatial-dynamics formulation permits the application of
dynamical-systems methods, in particular the centre-manifold reduction, and thus yields new insights into the problem; we refer
to~\cite{HaragusIooss11,VanderbauwhedeIooss92,IoossAdelmeyer98} for extensive material on that.

The far-reaching spatial-dynamics approach was later also applied 
in studying the stability of waves.
In~\cite{HaragusScheel02}, Haragus and Scheel used such a formulation
to prove spectral stability of capillary-gravity surface waves.
The idea of our spatial-dynamics approach to stability is close to theirs in principle
but differs at prominent places, notably in the lack of a finite-dimensional centre manifold.

When focussing on internal waves of sufficiently small amplitude, it is well known that
solitary waves are, to leading order, modelled by solitons of the Korteweg-deVries
equation, see~\cite{Benney66,Benjamin66,KirchgassnerLankers93,James97}. These KdV solitons exhibit a remarkable
stability as shown 
by Benjamin~\cite{Benjamin72}, by Bona et al.~\cite{BonaSouganidisStrauss87}, and by Pego and
Weinstein~\cite{PegoWeinstein92,PegoWeinstein94}.
In~\cite{PegoWeinstein92}, Pego and Weinstein developed a unified framework for the stability of solitons for a class of
Hamiltonian PDEs by relating conserved quantities of a given soliton to properties of the Evans function associated with
it and this work proves, in particular, the spectral stability of KdV solitons 
that we will exploit in the small-amplitude limit.

Even though this paper passes only formally by the infinite-dimensional spatial-dynamics setting,
the large body of work by Latushkin and
collaborators towards infinite-dimensional Evans functions (see~\cite{GesztesyLatushkinEtAl08,LatushkinPogan11} and
references therein) has been a prime motivation, and still is.

We finally remark that the interest in internal solitary waves is also indicated by many publications dealing with
algorithms for their quantitative computation (e.g.\ see~\cite{TurkingtonEydelandEtAl91,StastnaLamb02,KingCarrEtAl11}).
These methods make it possible to investigate properties of internal waves numerically. 
It is a stimulating question whether numerically observed instabilities of ISWs, such as those reported in
\cite{StastnaLamb02,CarrKingEtAl11,PreusseFreistuehlerEtAl12}, can be captured by their spectral properties, 
i.e., by the emergence of unstable eigenvalues.

\textbf{Note.} This paper presents central results of the author's PhD thesis~\cite{Klaiber13-Diss}.

\section{Statement of the results}
In the mathematical modelling of internal waves, it is common practice to consider a two-dimensional channel,
\begin{displaymath}
	\mathfrak{C} = \left\{ (x,y): x\in\R, 0< y < 1 \right\},
\end{displaymath}
which is entirely filled with a non-homogeneous, inviscid, incompressible fluid, 
with the density stratification of the fluid at rest being given by a known differentiable function $\bar{\rho}(y)$ satisfying 
$\bar{\rho}(y)>0$ and $\bar{\rho}'(y)<0$ for all $y\in [0,1]$; such a $\bar{\rho}$ is called a \emph{stable
stratification}. 
A prototypical example is given by the exponential stratification, $\bar{\rho}(y)=\e^{-\delta y}$ with some fixed $\delta>0$, to which we
restrict from Thm.~II on.

The motion of the fluid is assumed to be governed 
by the Euler equations%
\footnote{In this paper, we do not enter any issues related to the very interesting general solution theory of the Euler
equations for heterogeneous incompressible fluids. See \cite{BeirValli80,Marsden76}.}%
, 
\begin{subequations}
\label{eq:euler-full}
\begin{align}
	\label{eq:euler-full-2} \rho_t + u\rho_x + v\rho_y &= 0,\\
	\label{eq:euler-full-3} \rho \left( u_t + u u_x+ v u_y \right) &= -p_x,\\
	\label{eq:euler-full-4} \rho \left( v_t + u v_x+ v v_y \right) &= -p_y-g\rho,
	\intertext{complemented by the incompressibility constraint}
	\label{eq:euler-full-1} u_x + v_y &= 0,
\end{align}
\end{subequations}
with $t$, $x$ and $y$ denoting the time, horizontal and vertical position, respectively, whereas the sought functions,
occasionally collected in the vector $U$, are given by density $\rho(t,x,y)$, velocity field $(u(t,x,y),v(t,x,y))$, and
pressure $p(t,x,y)$; the constant $g$ denotes acceleration due to gravity.

The requirement that the fluid cannot leave the domain $\mathfrak{C}$ is encoded in the boundary conditions
\begin{equation}
	\label{eq:boundarycondition}
	v(t,x,0) = 0 \quad \text{and}\quad v(t,x,1) = 0,
\end{equation}
the second of which is often referred to as the \textit{rigid lid} condition expressing the fact that in typical
applications free-surface displacements are negligible. 

In the present setting, the quiescent state $\bar{U}(y) = (\bar{\rho}(y),\bar{u}(y),\bar{v}(y),\bar{p}(y))^\T$ with 
\begin{equation*}
	\bar{u}(y) = 0,\quad \bar{v}(y) =0,\quad \bar{p}(y)=-g\int_{0}^{^y} \bar{\rho}(\eta)\d{\eta}
\end{equation*}
is, in fact, a stationary solution of the Euler equations \eqref{eq:euler-full} and satisfies
\eqref{eq:boundarycondition}.
This permits to precisely define the waves of interest: 
A function $U^c(\xi,y)$ is called an \emph{internal solitary wave (ISW)} of speed $c$ if
\begin{equation*}
	U(t,x,y) = U^c(x-ct,y)
\end{equation*}
is a classical solution of \eqref{eq:euler-full} satisfying \eqref{eq:boundarycondition} and if the wave profile
$U^c(\xi,y)$ tends to the quiescent state $\bar{U}(y)$ uniformly in $y$ as $\abs{\xi}\to\infty$.

The inception of rigorous mathematical investigations on internal waves is indicated by the remarkable observation
due to Dubreil-Jacotin in \cite{Dubreil-Jacotin37} and Long in \cite{Long53}
that travelling wave profiles can be found by solving a single nonlinear elliptic equation (for the stream function),
the \textit{Dubreil-Jacotin-Long equation}, depending parametrically on the wave speed $c$; for a related equation, 
see~\cite{Yih60}. 

Based on these equations, proofs for the existence of periodic and solitary internal travelling waves have been given via different
methods in particular through bifurcation theory for elliptic equations (see \cite{Turner81, Amick84}), 
by the use of various variational principles (see \cite{Benjamin66, BonaBoseEtAl83,TurkingtonEydelandEtAl91,
LankersFriesecke97}), and, for small-amplitude waves, via the spatial-dynamics approach due to Kirchgässner (see \cite{Kirchgassner82,
KirchgassnerLankers93,James97}).

\bigskip
In order to study the stability of ISWs, we consider the Euler eigenvalue problem.
Using the time-exponential perturbation
\begin{equation*}
	(\e^{\kappa t}\rho(\xi,y),
	\e^{\kappa t}u(\xi,y),
	\e^{\kappa t}v(\xi,y),
	\e^{\kappa t}p(\xi,y))^\T
\end{equation*}
in the linearization of \eqref{eq:euler-full} about an
ISW solution $U^c(\xi,y)$,
the eigenvalue problem reads
\begin{subequations}
	\label{eq:evp-euler-full}
	\begin{align}
-\kappa\rho &= (u^c-c)\rho_\xi + v^c\rho_y + u\rho^c_\xi + v\rho^c_y,\\
	-\rho^c \kappa u &=	\rho^c \left( (u^c-c)u_\xi + uu^c_\xi + v^cu_y + vu^c_y \right)\\
\notag	&\quad + \rho \left( (u^c-c)u^c_\xi + v^cu^c_y \right) +p_\xi,\\ 
	-\rho^c \kappa v &= \rho^c \left( (u^c-c)v_\xi + uv^c_\xi + v^cv_y + vv^c_y \right)\\
	\notag &\quad + \rho \left( (u^c-c)v^c_\xi + v^cv^c_y \right)  + p_y + g\rho,\\
	0 &= u_\xi + v_y.
	\end{align}
\end{subequations}
A number $\kappa\in\C$ with $\Re\kappa >0$ is called an \textit{unstable eigenvalue} if this system possesses a
bounded solution for the given $\kappa$. To show spectral stability of the ISWs, we have to exclude unstable
eigenvalues. We are aware that the meaning of \textit{eigenvalue} is vague here as we do not provide a
functional-analytic framework; however, there is a clear meaning for the truncated systems which form the core of
our approach.

The present work is divided into three parts corresponding with Secs.~2, 3, and 4: (1)~establishing, at a formal level, a spatial-dynamics formulation for the eigenvalue
problem, and corresponding finite-dimensional truncations (Thms.~I and~II); 
(2)~showing, rigorously, that the latter are amenable to the Evans function method (Thm.~III);
(3)~proving the absence of unstable modes close to the origin in the small-amplitude limit (Thm.~IV).

Now, to explain our results, let $\psi(\xi,y)$ denote the stream function associated with the
linearized velocity field, 
\begin{equation*}
 \psi_\xi=-v \text{ and } \psi_y=u.
\end{equation*}
We will consider the underlying space
\begin{equation*}
\mathcal{W} = \mathcal{L}^2(0,1)\times \mathcal{L}^2(0,1)\times \mathcal{L}^2(0,1)\times \mathcal{L}^2(0,1)
\end{equation*}
endowed with the scalar product 
\begin{equation}
	\label{eq:scalarproduct}
	\scalarprod{U,V} := \int_{0}^{1} (-\bar{\rho}') \left( \frac{U_1V_1}{(\bar{\rho}')^2} + U_2V_2+U_3V_3 + U_4V_4 \right)
	\d{y}.
\end{equation}
Since we assume $\bar{\rho}'<0$ on the closed interval $[0,1]$, $\scalarprod{\cdot,\cdot}$ is obviously equivalent to the
standard scalar product and hence $\left( \mathcal{W},\scalarprod{\cdot,\cdot} \right)$ is a Hilbert space.
Our fist (formal!) result will be as follows.
\begin{thm}[abbreviated statement]
Given an 
ISW $U^c(\xi,y)$, the associated eigenvalue problem \eqref{eq:evp-euler-full}
can be written as an abstract ordinary differential equation, posed in $\mathcal{W}$,
of the form
\begin{equation}
	\label{eq:evp-spatdyn-intro}
	\tag{E}
	W'(\xi) = \A(\xi;\kappa)W(\xi),
\end{equation}
in which the dependent variable assumes, at ``time'' $\xi$, a value
\begin{equation}
	W(\xi) = \left(\rho(\xi,\cdot),\psi(\xi,\cdot),\psi_\xi(\xi,\cdot),\psi_{\xi\xi}(\xi,\cdot)	\right)^\T\in \mathcal{W} 
	\label{eq:EntriesOfW}
\end{equation}
and the coefficient $\A$ is of the form
\begin{equation}
	\A(\xi;\kappa) = 
	\begin{pmatrix}
		R_1 & R_2 & R_3 & 0\\
		0 & 0 & 1 & 0\\
		0 & 0 & 0 & 1\\
		S_1 & S_2 & S_3 & S_4
	\end{pmatrix},
	\label{eq:EntriesOfA}
\end{equation}
where $R_1,\dots,S_4$ denote appropriate linear operators in $\mathcal{L}^2(0,1)$.
\end{thm}

This theorem is of central importance for our approach 
as it opens the door for 
the theory of dynamical systems.

Since the equation \eqref{eq:evp-spatdyn-intro} represents an abstract ODE on a Hilbert space of infinite dimension, 
we propose to consider finite-dimensional truncations of \eqref{eq:evp-spatdyn-intro}
in the spirit of \cite{GesztesyLatushkinEtAl08,OhSandstede10,LedouxMalhamEtAl09}. 
Having found a suitable Hilbert basis $\mathfrak{B}$ of $\mathcal{W}$, 
we are able to formulate formal Galerkin-type approximants; this is the content of
Thm.~\ref{thm:existenceoftruncations}.

\begin{thm}[abbreviated statement]
There exists a natural sequence of finite-dimensional truncations 
\begin{equation}
	\hat{W}'_N(\xi) = \hat{\A}_N(\xi;\kappa)\hat{W}_N(\xi),\quad \text{for $N=0,1,2,\dots$}
	\label{eq:evp-spatdyn-N-hat}
\end{equation}
of (E) such that the operator $\hat{\A}_N(\xi;\kappa)$ has the following matrix representation in the basis~$\mathfrak{B}$:
\begin{equation*}
  \NA(\xi;\kappa) = 
  \begin{pmatrix}
  \mathcal{A}_{0} + \mathcal{B}_{0,0} & \mathcal{B}_{0,1} & \cdots & \mathcal{B}_{0,N}\\
  \mathcal{B}_{1,0} & \mathcal{A}_{1} + \mathcal{B}_{1,1} & \cdots & \mathcal{B}_{1,N}\\
  \vdots & \vdots & \ddots & \vdots\\
  \mathcal{B}_{N,0} & \mathcal{B}_{N,1} & \cdots & \mathcal{A}_{N} + \mathcal{B}_{N,N}
  \end{pmatrix}
\end{equation*}
where, with $\lambda\equiv \frac{g}{c^2}$,
\begin{equation}
	\mathcal{A}_M = \mathcal{A}_M(\kappa) = 
	\begin{pmatrix}
		\frac{\kappa}{c} & 0 & -\frac{1}{c} & 0\\
		0 & 0 & 1 & 0\\
		0 & 0 & 0 & 1\\
		\delta\lambda\kappa & -\frac{1}{c}\kappa\lambda_M\delta & \delta (\lambda_M-\lambda) &\frac{\kappa}{c}
	\end{pmatrix}
	\label{eq:matrixAM}
	\end{equation}
	and $\mathcal{B}_{ML} = \mathcal{B}_{ML}(\xi;\kappa)$ with 
	\begin{equation*}
 \mathcal{B}_{ML}(\pm\infty;\kappa)=0.
\end{equation*}
\end{thm}
The representation of system \eqref{eq:evp-spatdyn-N-hat} in the basis~$\mathfrak{B}$ will be denoted by E$_N$.
The treatment of the truncations (E$_N$) by an Evans function approach is considered next.

The \textit{Evans function} is a tool to detect the point spectrum of differential operators. The main idea is
that under certain assumptions eigenvalues of the linearized operator can be found as the roots of this analytic function. 
Originally defined for travelling waves in reaction-diffusion equations, the Evans function was later substantially
extended to cover viscous conservation laws and 
dispersive equations as well, see
\cite{Evans75,AlexanderGardnerEtAl90,PegoWeinstein92,GardnerZumbrun98,FreistuhlerSzmolyan10}
and references therein; we refer to~\cite{Sandstede02} for an extensive introduction. 

That the finite-dimensional ODE systems (E$_N$) are
amenable to standard Evans function theory, is accomplished in Thm.~\ref{thm:existenceofevansfunctions}.

\begin{thm}[abbreviated statement]
Consider a regular ISW of some speed $c>c_0$ and the associated truncated problem \eqref{eq:EN} for an arbitrary $N\in\N$.
Then, there exist an open domain $\Omega=\Omega(N,c)\subset\C$ comprising the closed right half-plane $\Cpluscl$
and an analytic function $D_N:\Omega\to\C$
such that \textnormal{(E$_N$)} has a bounded solution for $\kappa\in\Cplus$ if and only if $D_N(\kappa) = 0$.
\end{thm}

The guiding idea in the background is that sequences of appropriately scaled (`truncated eigen-') functions $\hat{W}_N$ and
Evans functions $D_N$ converge, as $N\to\infty$, to eigenfunctions $\hat{W}$ and an Evans function $D$ of the original
infinite-dimensional problem (E), and $\hat{W}$ yields a solution to \eqref{eq:evp-euler-full}. 
Important as they are, these issues are not yet considered in the present paper. 

Finally, we consider the small-amplitude limit. In that case, we are able to prove the absence of zeros of $D_N$ in a
certain neighbourhood of the origin.
\begin{thm}
	For all $N\in\N$ and any $R_0>0$ there exists some $\varepsilon_0>0$ such that for all $\varepsilon\in(0,\varepsilon_0)$
  the Evans function $D_{N,\varepsilon}(\kappa)$ associated with an ISW of amplitude $\varepsilon^2$ satisfies
	\begin{equation*}
		D_{N,\varepsilon}(0) =0,\quad 	D_{N,\varepsilon}'(0) = 0 
	\end{equation*}
	and
	\begin{equation*}
		D_{N,\varepsilon}(\kappa) \not= 0 \text{ for all $\kappa\in\Cpluscl\setminus\{0\}$ with $\abs{\kappa} < R_0\varepsilon^{3}$.}
	\end{equation*}
\end{thm}

\setcounter{thm}{0}

\section{The spatial-dynamics formulation and its truncations}
\label{sec:PfThm1}
In the present section, we prove the announced formulation of the eigenvalue problem as a formal spatial-dynamical
system. 
In the whole section, we systematically disregard questions of regularity, domains and ranges of operators, etc.
However, we make the following precise assumptions on the profile $U^c$.\\
(A1)~Differentiability: The profile satisfies $U^c\in C^3(\mathfrak{C})$.\\
(A2)~Exponential decay: There are constants $C_1,C_2>0$ such that 
\begin{equation*}
	\abs{ \frac{\partial^{\alpha+\beta}}{\partial\xi^\alpha\partial y^\beta}\left( U^c(\xi,y)-\bar{U}(y) \right)} \le C_1\e^{-C_2 \abs{\xi}}
\end{equation*}
for all $\alpha,\beta\in\{0,1,2,3\}$ with $0\le \alpha+\beta\le 3$.\\
(A3)~Monotonicity: Each level set $\{(\xi,y):\rho^c(\xi,y) = \varrho\}$, $\varrho\in\text{range}\bar{\rho}$,
can be written as the graph $y=Y^\varrho(\xi)$ of some differentiable function $Y^\varrho:\R\to[0,1]$.

By a \emph{regular ISW} we mean an ISW satisfying (A1), (A2) and (A3). Note that for a regular ISW
$Y^\varrho(\pm\infty)=\bar{\rho}^{-1}(\varrho)$ because of (A2).
These assumptions are natural to make since they are satisfied for small-amplitude waves (which follows from the
explicit description of such waves, e.g. in \cite{James97}) and since 
the numerical results reported in \cite{TurkingtonEydelandEtAl91,LambWan98} suggest that they indeed hold far beyond the
small-amplitude regime. 

We introduce stream functions $\psi^c$ and $\psi$ associated with the velocity field $(u^c,v^c)$
of the travelling wave and the linearized velocity field $(u,v)$: 
The incompressibility constraint \eqref{eq:euler-full-1} in the problem's original formulation implies that the vector fields $(-v^c,u^c)$ and $(-v,u)$,
defined on the simply-connected domain $\mathfrak{C}$, possess potentials $\psi^c$ and $\psi$ satisfying
the relations
\begin{equation*}
\psi^c_\xi = -v^c,
\psi^c_y = u^c
\quad \text{and} \quad 
  \psi_\xi = -v,
	\psi_y = u
\end{equation*}
and, in view of (A2) and \eqref{eq:boundarycondition}, the boundary conditions
\begin{equation*}
\psi^c(\xi,0) = \psi^c(\xi,1) = 0
\quad \text{and} \quad 
\psi(\xi,0) = \psi(\xi,1) = 0.
\end{equation*}
The use of the stream function is motivated by Benjamin's version of the Euler equations (see \cite[p.\
34ff.]{Benjamin84}).%
\footnote{%
How this formulation underlies the present approach is explained in more detail in the author's PhD
thesis~\cite[Ch.\ 2]{Klaiber13-Diss}.}

As already noted in the introduction, the underlying function space is the Hilbert space
$\mathcal{W} = \left( \mathcal{L}^2(0,1) \right)^4$ with the scalar product $\scalarprod{\cdot,\cdot}$ from \eqref{eq:scalarproduct}.

\begin{thm}
\label{thm:evp-spatdyn}
Given a regular ISW $U^c(\xi,y)$, the associated eigenvalue problem 
\eqref{eq:evp-euler-full}
can formally be written as the abstract ordinary differential equation 
\begin{equation}
	W'(\xi) = \A(\xi;\kappa)W(\xi),
	\tag{E}
	\label{eq:evp-spatdyn}
\end{equation}
in the space $(\mathcal{W},\scalarprod{\cdot,\cdot})$ 
where 
\begin{equation}
	W(\xi) = \left(\rho(\xi,\cdot),\psi(\xi,\cdot),\psi_\xi(\xi,\cdot),\psi_{\xi\xi}(\xi,\cdot)	\right)^\T\in \mathcal{W} 
	\label{eq:EntriesOfW}
\end{equation}
and 
\begin{equation}
	\A(\xi;\kappa) = 
	\begin{pmatrix}
		R_1 & R_2 & R_3 & 0\\
		0 & 0 & 1 & 0\\
		0 & 0 & 0 & 1\\
		S_1 & S_2 & S_3 & S_4
	\end{pmatrix}
	\label{eq:EntriesOfA}
\end{equation}
with 
\begin{equation*}
  R_j =\frac{\tilde{R}_j}{\psi^c_{y}-c} 
  \quad
  \text{and}
  \quad
	S_k = \frac{\tilde{S}_k}{(\psi^c_{y}-c)\rho^c}
\end{equation*}
for $j\in\{1,2,3\}$ and $k\in\{1,2,3,4\}$, where 
\begin{align*}
  \tilde{R}_1 &=-\kappa + \psi^c_{\xi}\partial_y,\\
  \tilde{R}_2 &=-\rho^c_{\xi}\partial_y,\\
  \tilde{R}_3 &= \rho^c_{y},\\
\tilde{S}_1 &= 
\big[-(\psi^c_y-c)\left(\psi^c_{\xi yy}+\psi^c_{\xi\xi\xi} \right) +
		\psi^c_\xi\left(\psi^c_{yyy}+\psi^c_{\xi\xi y}   \right)\\&\quad  - (\psi^c_y-c)^{-1}
		\{g\kappa + \kappa\psi^c_\xi\psi^c_{\xi y}\}+ \kappa\psi^c_{\xi\xi}\big]\\
	&\quad + \big[-(\psi^c_y-c)\psi^c_{\xi y} + \psi^c_\xi\psi^c_{yy} - \psi^c_\xi\psi^c_{\xi\xi} +
		(\psi^c_y-c)^{-1}(g\psi^c_\xi + (\psi^c_\xi)^2\psi^c_{\xi y})\big]	\partial_y,\\
	\tilde{S}_2 &= \big[-\rho^c_y\psi^c_{\xi y}-\rho^c\psi^c_{\xi yy}-\rho^c\psi^c_{\xi\xi\xi} -(\psi^c_y-c)^{-1}
		(g\rho^c_\xi+\rho^c_\xi\psi^c_\xi\psi^c_{\xi y})-\kappa\rho^c_y\big]\partial_y \\	&\quad 
		+ \big[\rho^c_y\psi^c_\xi-\kappa\rho^c\big]\partial_{yy} 
		+ \big[\rho^c\psi^c_\xi\big]\partial_{yyy},\\
	\tilde{S}_3 &=
	\big[\rho^c_y\psi^c_{yy}+\rho^c_\xi\psi^c_{\xi y}+\rho^c\psi^c_{yyy}+\rho^c\psi^c_{\xi\xi
		y}-\rho^c_y\psi^c_{\xi\xi}\\&\quad + (\psi^c_y-c)^{-1}(g\rho^c_y+\rho^c_y\psi^c_\xi\psi^c_{\xi y}) -
		\kappa\rho^c_\xi\big]\\
	&\quad + [-\rho^c_y(\psi^c_y-c)+\rho^c_\xi\psi^c_\xi] \partial_y 
	+ [-\rho^c(\psi^c_y-c)] \partial_{yy},\\
	\tilde{S}_4 &= [-\rho^c_\xi(\psi^c_y-c)-\kappa\rho^c] + [\rho^c\psi^c_\xi] \partial_y.
\end{align*}
\end{thm}
\begin{proof}
Using $\psi^c(\xi,y)$ and $\psi(\xi,y)$, system \eqref{eq:evp-euler-full} assumes the form
	\begin{subequations}
		\label{eq:evp-euler-full-streamfunction}
		\begin{align}
		\label{eq:evp-euler-full-streamfunction-rho}
			-\kappa\rho &= (\psi^c_y-c)\rho_\xi -\psi^c_\xi\rho_y + \rho^c_\xi\psi_y - \rho^c_y\psi_\xi\\
		\label{eq:evp-euler-full-streamfunction-u}
			-\kappa\rho^c\psi_y &=
			\rho^c \left( (\psi^c_y-c)\psi_{\xi y} + \psi^c_{\xi y} \psi_y -\psi^c_\xi \psi_{yy} -\psi^c_{yy} \psi_\xi\right)\\
			\notag
			&\quad + \rho \left( (\psi^c_y-c)\psi^c_{\xi y} -\psi^c_\xi\psi^c_{yy} \right) +p_\xi,\\
		\label{eq:evp-euler-full-streamfunction-v}
			\kappa \rho^c \psi_\xi &= \rho^c \left( -(\psi^c_y-c)\psi_{\xi\xi} - \psi^c_{\xi\xi} \psi_y 
			+ \psi^c_\xi \psi_{\xi y}+ \psi^c_{\xi y} \psi_\xi\right)\\
			\notag
			&\quad + \rho \left( -(\psi^c_y-c)\psi^c_{\xi\xi} + \psi^c_\xi\psi^c_{\xi y} \right)  + p_y + g\rho.
		\end{align}
	\end{subequations}
	In terms of the variables
	\begin{equation*}
		W_1 = \rho,\; W_2 = \psi,\; W_3 = \psi_\xi,\; W_4 = \psi_{\xi\xi},
	\end{equation*}
	Eq.~\eqref{eq:evp-euler-full-streamfunction-rho}	becomes
	\begin{equation*}
		(\psi^c_{y}-c) W_1' = (-\kappa + \psi^c_{\xi}\partial_y) W_1 - \rho^c_{\xi}\partial_y W_2 + \rho^c_{y} W_3, 
	\end{equation*}
	from which we read off the expressions $R_1,R_2,R_3$.

	To eliminate the pressure from Eqs.~\eqref{eq:evp-euler-full-streamfunction-u},
	\eqref{eq:evp-euler-full-streamfunction-v} we consider the equation 
	``$\partial_y$\eqref{eq:evp-euler-full-streamfunction-u}$-\partial_\xi$\eqref{eq:evp-euler-full-streamfunction-v}''.
	On the left hand side, we obtain
	\begin{equation*}
		\text{LHS} = -\kappa \left( (\rho^c\psi_\xi)_\xi + (\rho^c\psi_y)_y \right)
	\end{equation*}
  and for the right hand side: 
	\begin{align*}
		\text{RHS} 
		&= \rho^c_y\left( (\psi^c_y-c)\psi_{\xi y} + \psi^c_{\xi y}\psi_y-\psi^c_\xi\psi_{yy}-\psi^c_{yy}\psi_\xi\right)\\
		&\quad + \rho^c_\xi\left( (\psi^c_y-c)\psi_{\xi\xi} + \psi^c_{\xi\xi}\psi_y-\psi^c_\xi\psi_{\xi y}-\psi^c_{\xi y}\psi_\xi\right)\\
		&\quad + \rho^c \left( (\psi^c_y-c)\psi_{\xi yy} + \psi^c_{\xi yy}\psi_y-\psi^c_\xi\psi_{yyy}-\psi^c_{yyy}\psi_\xi\right)\\
		&\quad + \rho^c \left( (\psi^c_y-c)\psi_{\xi\xi\xi} + \psi^c_{\xi\xi\xi}\psi_y-\psi^c_\xi\psi_{\xi\xi
		y}-\psi^c_{\xi\xi y}\psi_\xi\right)\\
		&\quad + \rho_y \left( (\psi^c_y-c)\psi^c_{\xi y} -\psi^c_\xi\psi^c_{yy}\right)\\
		&\quad + \rho_\xi \left( (\psi^c_y-c)\psi^c_{\xi\xi} -\psi^c_\xi\psi^c_{\xi y}-g\right)\\
		&\quad + \rho \left[ \left( (\psi^c_y-c)\psi^c_{\xi yy} -\psi^c_\xi\psi^c_{yyy}\right) + 
		\left( (\psi^c_y-c)\psi^c_{\xi\xi\xi} -\psi^c_\xi\psi^c_{\xi\xi y}\right)\right].
	\end{align*}
	By replacing $\rho_\xi = W_1' = R_1W_1 + R_2W_2 + R_3W_3$ and solving for $\psi_{\xi\xi\xi}$,
	we obtain an equation of the form
	\begin{align*}
		\rho^c(\psi^c_y-c)\psi_{\xi\xi\xi} &=
		\left( \tilde{S}_1^0 + \tilde{S}_1^1\partial_y \right) W_1
		+ \left( \tilde{S}_2^1\partial_y + \tilde{S}_2^2\partial_y^2 + \tilde{S}_2^3\partial_y^3 \right) W_2\\
		&\quad + \left( \tilde{S}_3^0 + \tilde{S}_3^1\partial_y^1 + \tilde{S}_3^2\partial_y^2 \right) W_3
		+ \left( \tilde{S}_4^0 + \tilde{S}_4^1\partial_y \right) W_4
	\end{align*}
	where the functions $\tilde{S}_k^j$ are given as follows.
	\begin{align*}
		\tilde{S}_1^0 &= -(\psi^c_y-c)\left(\psi^c_{\xi yy}+\psi^c_{\xi\xi\xi} \right) +
		\psi^c_\xi\left(\psi^c_{yyy}+\psi^c_{\xi\xi y}   \right)\\
		&\quad - (\psi^c_y-c)^{-1}\{g\kappa + \kappa\psi^c_\xi\psi^c_{\xi y}\}+ \kappa\psi^c_{\xi\xi},\\
		\tilde{S}_1^1 &= -(\psi^c_y-c)\psi^c_{\xi y} + \psi^c_\xi\psi^c_{yy} - \psi^c_\xi\psi^c_{\xi\xi} +
		(\psi^c_y-c)^{-1}(g\psi^c_\xi + (\psi^c_\xi)^2\psi^c_{\xi y}),\\ 
		\tilde{S}_2^1 &= -\rho^c_y\psi^c_{\xi y}-\rho^c\psi^c_{\xi yy}-\rho^c\psi^c_{\xi\xi\xi}-(\psi^c_y-c)^{-1}
		(g\rho^c_\xi+\rho^c_\xi\psi^c_\xi\psi^c_{\xi y})-\kappa\rho^c_y,\\
		\tilde{S}_2^2 &= \rho^c_y\psi^c_\xi-\kappa\rho^c,\\
		\tilde{S}_2^3 &= \rho^c\psi^c_\xi,\\
		\tilde{S}_3^0 &= \rho^c_y\psi^c_{yy}+\rho^c_\xi\psi^c_{\xi y}+\rho^c\psi^c_{yyy}+\rho^c\psi^c_{\xi\xi
		y}-\rho^c_y\psi^c_{\xi\xi}\\
		&\quad + (\psi^c_y-c)^{-1}(g\rho^c_y+\rho^c_y\psi^c_\xi\psi^c_{\xi y}) - \kappa\rho^c_\xi,\\
		\tilde{S}_3^1 &= -\rho^c_y(\psi^c_y-c)+\rho^c_\xi\psi^c_\xi,\\
		\tilde{S}_3^2 &= -\rho^c(\psi^c_y-c),\\
		\tilde{S}_4^0 &= -\rho^c_\xi(\psi^c_y-c)-\kappa\rho^c,\\
		\tilde{S}_4^1 &= \rho^c\psi^c_\xi.
	\end{align*}
	This yields the asserted expressions for the operators
	\begin{align*}
	\tilde{S}_1 &= \tilde{S}_1^0 + \tilde{S}_1^1\partial_y,\\
	\tilde{S}_2 &= \tilde{S}_2^1\partial_y + \tilde{S}_2^2\partial_y^2 + \tilde{S}_2^3\partial_y^3,\\
	\tilde{S}_3 &= \tilde{S}_3^0 + \tilde{S}_3^1\partial_y^1 + \tilde{S}_3^2\partial_y^2,\\
	\tilde{S}_4 &= \tilde{S}_4^0 + \tilde{S}_4^1\partial_y.
	\end{align*}
	We finally mention that the operators $R_i$ and ${S}_j$ are not singular.
	As $U^c$ is assumed to be a regular ISW, hypothesis (A1) ensures that $\rho^c$ and $\psi^c$ possess the
	differentiability which is required to make sense of the expressions $R_i$ and $\tilde{S}_j$.
	It is well-known (see \cite[p.\ 98f.]{TurkingtonEydelandEtAl91}) that $\rho^c(\xi,y)$ and $\psi^c(\xi,y)-cy$ are
  related by $\rho^c(\xi,y)=\bar{\rho}(-\frac{1}{c}(\psi^c(\xi,y)-cy))$,
	hence hypothesis (A3) ensures that the denominators do not vanish: $0<\bar{\rho}(1)\le \rho^c(\cdot,\cdot)\le \bar{\rho}(0)$
	and
	${\partial_y}(\psi^c-cy) = -c \bar{\rho}'(-\frac{1}{c}(\psi^c(\xi,y)-cy))^{-1}{\partial_y}\rho^c(\xi,y)\not=0$ (due to local solvability).
\end{proof}

Next, we turn to the derivation of finite-dimensional truncations of \eqref{eq:evp-spatdyn} by formally applying a
Galerkin-type procedure. 
It is well-known (see \textnormal{\cite[p.\ 250ff.]{Benjamin66}}) that the operator 
\label{def:phi0}
	\begin{equation*}
		\frac{1}{\bar{\rho}'} T=\frac{1}{\bar{\rho}'}\partial_y(\bar{\rho}\partial_y):
		H^2(0,1)\cap H^1_0(0,1) \subseteq \mathcal{L}_{-\bar{\rho}'}^2(0,1)\to \mathcal{L}_{-\bar{\rho}'}^2(0,1),
	\end{equation*}
	with zero boundary conditions, is self-adjoint, positive and uniformly elliptic; its spectrum therefore consists of 
	a sequence of real eigenvalues $0<\lambda_0<\lambda_1<\dots \to \infty$ and there 
	exists an orthonormal basis for $\mathcal{L}_{-\bar{\rho}'}^2(0,1)$
	of corresponding eigenfunctions $\{\varphi_n\}_{n\in\N}$, 
	which are normalized and mutually orthogonal, i.e.
	\begin{equation}
		\int_{0}^{1} (-\bar{\rho}') \varphi_n\varphi_m \d{y} = \delta_{nm}.
		\label{eq:scalarprod-vertical}
	\end{equation}
  It is obvious now that the scalar product \eqref{eq:scalarproduct} on the space $\mathcal{W} = \left(
  \mathcal{L}^2(0,1) \right)^4$ is motivated from this scalar product \eqref{eq:scalarprod-vertical}.
  What is more, we can find an explicit, particularly suitable Hilbert basis 
for $\mathcal{W}$ emanating from the set $\{\varphi_n\}_{n\in\N}$ of eigenfunctions of $T$.
For all $N\in\N$, we set
\begin{equation*}
	U_N^1 = \begin{pmatrix} \bar{\rho}' \varphi_N \\ 0 \\ 0 \\ 0	\end{pmatrix},\,
	U_N^2 = \begin{pmatrix} 0 \\ \varphi_N \\ 0 \\ 0	\end{pmatrix},\,
	U_N^3 = \begin{pmatrix} 0 \\ 0 \\ \varphi_N \\ 0	\end{pmatrix},\,
	U_N^4 = \begin{pmatrix} 0 \\ 0 \\ 0 \\ \varphi_N 	\end{pmatrix}.
\end{equation*}
Then, as a direct consequence of $\{\varphi_n\}_{n\in\N}$ being a Hilbert basis for
$\mathcal{L}^2_{-\bar{\rho}'}(0,1)$, the set 
  \begin{equation*}
    \mathfrak{B} := \{U_N^k: k\in\{1,2,3,4\}, N\in\N\}
  \end{equation*}
  forms a Hilbert basis of the Hilbert space $(\mathcal{W},\scalarprod{\cdot,\cdot})$.

For the sake of concreteness, we restrict -- as it was done too in \cite{LankersFriesecke97,BenneyKo78} -- to the
important special case of an exponential stratification,
\begin{equation}
	\bar{\rho} (y) = \exp(-\delta y),\quad \text{with fixed $\delta>0$.}
	\label{eq:rhoisexp}
\end{equation}
The rest of this section serves to prove the following theorem.
\begin{thm}
	\label{thm:existenceoftruncations}
  With respect to the Hilbert basis $\mathfrak{B}$, the infinite-dimensional spatial-dynamics formulation
  \eqref{eq:evp-spatdyn} has, for any $N\in\N$, a formal finite-dimensional truncation
			\begin{equation}
	\begin{pmatrix}
		w_0 \\ w_1 \\ \vdots \\ w_N
	\end{pmatrix}_\xi =
	\NA(\xi;\kappa)
	\begin{pmatrix}
		w_0 \\ w_1 \\ \vdots \\ w_N
	\end{pmatrix},
				\label{eq:EN}
				\tag{E$_N$}
			\end{equation}
			with 
			\begin{equation*}
				\NA(\xi;\kappa) = \begin{pmatrix}
		\mathcal{A}_{0} + \mathcal{B}_{0,0} & \mathcal{B}_{0,1} & \cdots & \mathcal{B}_{0,N}\\
		\mathcal{B}_{1,0} & \mathcal{A}_{1} + \mathcal{B}_{1,1} & \cdots & \mathcal{B}_{1,N}\\
		\vdots & \vdots & \ddots & \vdots\\
		\mathcal{B}_{N,0} & \mathcal{B}_{N,1} & \cdots & \mathcal{A}_{N} + \mathcal{B}_{N,N}
	\end{pmatrix}
			\end{equation*}
			where, with $\lambda\equiv \frac{g}{c^2}$,
			\begin{equation}
				\mathcal{A}_M = \mathcal{A}_M(\kappa) = 
				\begin{pmatrix}
			\frac{\kappa}{c} & 0 & -\frac{1}{c} & 0\\
			0 & 0 & 1 & 0\\
			0 & 0 & 0 & 1\\
			\delta\lambda\kappa & -\frac{1}{c}\kappa\lambda_M\delta & \delta (\lambda_M-\lambda) &\frac{\kappa}{c}
		\end{pmatrix}
		\label{eq:matrixAM}
			\end{equation}
			and $\mathcal{B}_{ML} = \mathcal{B}_{ML}(\xi;\kappa)$ with 
			\begin{equation*}
				 \mathcal{B}_{ML}(\pm\infty;\kappa)=0.
			\end{equation*}
\end{thm}

\begin{proof}

We define
\begin{equation*}
	X_M := \spn \{ U_M^1, U_M^2,U_M^3, U_M^4 \} \text{  and  } \mathcal{W}_N := \bigoplus_{0\leq M\leq N} X_M,
\end{equation*}
and denote by $Q_N$ the orthogonal projection onto the space $\mathcal{W}_N$. 
Clearly, $\left( \mathcal{W}_N \right)_{N\in\N}$
is an increasing sequence of subspaces $\mathcal{W}_N\subset \mathcal{W}$ of dimensions 
\begin{equation*}
	d_N := 4N+4<\infty.
\end{equation*}
From now on, we fix some $N\in\N$.
Starting from \eqref{eq:evp-spatdyn}, the announced truncated problem \eqref{eq:EN} is obtained in two steps:
In step one, we replace $\mathbb{A}$ by its projected version $Q_N\mathbb{A}Q_N$ and 
in step two, we write down the representation of the projected problem in the basis~$\mathfrak{B}$. 

\textit{Step 1}: We split the linear operator $\A$ in two parts,
\begin{equation}
	\A(\xi;\kappa) = \A^\infty(\kappa) + \B(\xi;\kappa),
	\label{eq:splittingofA}
\end{equation}
with 
\begin{equation*}
	\A^\infty(\kappa) := \lim_{\xi\to\pm\infty} \A(\xi;\kappa)
	=	\begin{pmatrix}
		\frac{\kappa}{c} & 0 & -\frac{\bar{\rho}'}{c} & 0\\
		0 & 0 & 1 & 0\\
		0 & 0 & 0 & 1\\
		-\frac{g\kappa}{c^2 \bar{\rho}} & \frac{\kappa}{c \bar{\rho}} T & -\frac{1}{\bar{\rho}} \left(
		  T -\frac{g}{c^2}\bar{\rho}'\right) & \frac{\kappa}{c}
	\end{pmatrix}.
\end{equation*}
A direct calculation shows that
\begin{equation}
\begin{aligned}
	\A^\infty(\kappa)U_M^1 &= \frac{\kappa}{c}U_M^1 + \delta\lambda\kappa U_M^4,\\
	\A^\infty(\kappa)U_M^2 &= -\delta\lambda_M\frac{\kappa}{c} U_M^4,\\
	\A^\infty(\kappa)U_M^3 &= -\frac{1}{c}U_M^1 + U_M^2 + \delta(\lambda-\lambda_M)U_M^4,\\
	\A^\infty(\kappa)U_M^4 &= U_M^3 + \frac{\kappa}{c} U_M^4,
\end{aligned}
	\label{eq:AinftyUM}
\end{equation}
thus each $X_M$ with $1\le M\le N$ is invariant under $\A^{\infty}(\kappa)$; consequently, $\mathcal{W}_N$ is
$\A^{\infty}(\kappa)$-invariant as well.
As announced above, we construct finite-dimensional versions of \eqref{eq:evp-spatdyn} by 
substituting $Q_N\A Q_N$ for $\A$, so 
instead of \eqref{eq:evp-spatdyn} we now consider 
\begin{equation*}
	Q_NW'(\xi) = Q_N\left( \A^\infty(\kappa)+\B(\xi;\kappa) \right)Q_NW(\xi).
\end{equation*}
Since we know $\A^\infty(\kappa)\mathcal{W}_N\subset \mathcal{W}_N$,
we focus on the following ODE
\begin{equation}
	W_N'(\xi) = \left(\left.\A^\infty(\kappa)\right|_{\Im Q_N} + \left.Q_N \B(\xi;\kappa)\right|_{\Im Q_N} \right)W_N(\xi)
	\label{eq:truncatedproblem-WN}
\end{equation}
for $W_N = Q_NW\in \mathcal{W}_N$. 

\textit{Step 2}: In order to derive \eqref{eq:EN} from \eqref{eq:truncatedproblem-WN}, we expand {$W_N\in \mathcal{W}_N$} in the
basis $\mathfrak{B}$, i.e.\ 
\begin{equation*}
	W_N(\xi) = \sum_{M=0}^N \sum_{k=1}^4 w_M^k(\xi) U_M^k,
\end{equation*}
and 
introduce the notation 
\begin{equation*}
	\mathcal{A}^{lk}_{M} := \scalarprod{\A^\infty U_M^k,U_M^l}, \quad\mathcal{B}^{lk}_{ML} := \scalarprod{\B U_L^k,U_M^l}
\end{equation*}
for $M,L\in\{0,1,\dots, N\}$ and $l,k\in\{1,2,3,4\}$. 
We emphasize that, due to the splitting \eqref{eq:splittingofA}, we have
\begin{equation*}
	\mathcal{A}_{M} = \mathcal{A}_{M}(\kappa) \quad \text{and}\quad \mathcal{B}_{ML} = \mathcal{B}_{ML}(\xi;\kappa),
\end{equation*}
and notably 
\begin{equation*}
	\mathcal{B}_{ML}(\pm\infty;\kappa) = 0.
\end{equation*}
Furthermore, Eq.~\eqref{eq:AinftyUM} implies that
\begin{equation*}
\mathcal{A}_{M} = 
				\begin{pmatrix}
			\frac{\kappa}{c} & 0 & -\frac{1}{c} & 0\\
			0 & 0 & 1 & 0\\
			0 & 0 & 0 & 1\\
			\delta\lambda\kappa & -\frac{1}{c}\kappa\lambda_M\delta & \delta (\lambda-\lambda_M) &\frac{\kappa}{c}
		\end{pmatrix}
\end{equation*}
as claimed in \eqref{eq:matrixAM}.

In terms of this notation, we obtain 
\begin{equation}
	\begin{pmatrix}
		w_0 \\ w_1 \\ \vdots \\ w_N
	\end{pmatrix}_\xi =
	\underbrace{\begin{pmatrix}
		\mathcal{A}_{0} + \mathcal{B}_{0,0} & \mathcal{B}_{0,1} & \cdots & \mathcal{B}_{0,N}\\
		\mathcal{B}_{1,0} & \mathcal{A}_{1} + \mathcal{B}_{1,1} & \cdots & \mathcal{B}_{1,N}\\
		\vdots & \vdots & \ddots & \vdots\\
		\mathcal{B}_{N,0} & \mathcal{B}_{N,1} & \cdots & \mathcal{A}_{N} + \mathcal{B}_{N,N}
	\end{pmatrix}}_{=\NA}
	\begin{pmatrix}
		w_0 \\ w_1 \\ \vdots \\ w_N
	\end{pmatrix}
	\label{eq:EN-2}
\end{equation}
with $w_n = (w_n^0,w_n^1,w_n^2,w_n^3)^\T$.
\end{proof}

Note that the matrix $\NA(\xi;\kappa)$ 
decays -- due to hypothesis (A2) -- 
exponentially fast 
to the constant coefficient matrix
\begin{equation}
	\NA^\infty(\kappa) = \mathcal{A}_{0}(\kappa) \oplus \dots \oplus \mathcal{A}_{N}(\kappa).
	\label{eq:AinftyIsASum}
\end{equation} 
The spectrum of $\NA^\infty(\kappa)$ will be described in Lm.~\ref{lem:eigenvalues}. 

\section{Evans functions for the truncated eigenvalue problems}
\label{sec:TruncatedEvansFunctions}
In the present section, we construct an Evans function
for each of the finite-dimensional truncated versions \eqref{eq:EN}
of the problem (E) 
obtained in the second part of Sec.~2. 
The central result is the following theorem on the existence of Evans functions for detecting growing modes.
\begin{thm}
	\label{thm:existenceofevansfunctions}
  Given a regular ISW of speed $c>c_0$ and any $N\in\N$, there exists a well-defined, analytic Evans function 
	\begin{equation*}
		D_N:\Omega\to\C
	\end{equation*}
  that has the property
  
  \vspace{0.5cm}
  \centerline{\mbox{\eqref{eq:EN} possesses a bounded solution for $\kappa$
  \; iff \; $D_N(\kappa)=0$}}
  \vspace{0.5cm}

  for all $\kappa\in\C$ with $\Re\kappa>0$ and satisfies
	\begin{equation}
		D_N(0) = 0\quad \text{and}\quad D_N'(0) = 0.
		\label{eq:DoubleZeroOfD}
	\end{equation}
\end{thm}

This theorem is a consequence of the next lemma on the existence of analytic bundles corresponding to stable and
unstable spaces.
For brevity, we let $\mathcal{G}^n_k(\C)$ denote the Grassmannian of all $k$-dimensional subspaces of $\C^n$ and 
consider $\mathcal{G}^n_k(\C)$ with its standard structure as a compact complex-analytic manifold 
(see~\cite[p.\ 193ff.]{GriffithsHarris94}). 
\begin{lem}
  \label{lem:ExOfUandS}
  Consider a regular ISW of some speed $c>c_0$ and the associated truncated problem \eqref{eq:EN} for an arbitrary
			$N\in\N$. Let
			\begin{equation*}
				d_N^{s} := N+1 \text{ and } d_N^{u} := 3N+3.
			\end{equation*}
			Then, there exist an open domain $\Omega=\Omega(N,c)\subset\C$ comprising the closed right half-plane $\Cpluscl$ and
			complex-analytic mappings 
	\begin{align*}
		\mathcal{S}_N &:\Omega\to \mathcal{G}^{d_N}_{d^s_N}(\C), \\
		\mathcal{U}_N &:\Omega\to \mathcal{G}^{d_N}_{d^u_N}(\C), 
	\end{align*}
	such that the following characterization holds 
	for any $\kappa$ in the open right half-plane $\Cplus$ 
	and any solution $w:\R\to \C^{d_N}$ of \eqref{eq:EN}:
	\begin{align*}
	w(0)\in \mathcal{S}_N(\kappa) &\text{ iff } w(+\infty) = 0,\\
	\intertext{and}
	w(0)\in \mathcal{U}_N(\kappa) &\text{ iff } w(-\infty) = 0.
	\end{align*}
\end{lem}

\noindent For the rest of this section, ${}'$ always denotes the derivative with respect to $\kappa$.
The following observation is crucial for the proof of Lm.~\ref{lem:ExOfUandS}. 
\begin{lem}
	\label{lem:consistentsplitting}
	For any $N\in\N$ and any $\kappa\in\C$ with $\Re\kappa>0$, the matrix 
	\begin{equation*}
		\NA^\infty(\kappa)\in\C^{d_N\times d_N}
	\end{equation*}
	possesses $d_N^{s} := N+1$ eigenvalues with negative real part and $d_N^{u} := 3N+3$ eigenvalues  with positive real part.
\end{lem}

An immediate consequence of Lm.~\ref{lem:consistentsplitting} is that, for all $\kappa\in\Cplus$, 
the underlying space $\mathcal{W}_N\cong\C^{d_N}$ splits into a direct sum,
\begin{equation}
	\mathcal{W}_N = \mathcal{S}_N^\infty(\kappa) \oplus \mathcal{U}_N^\infty(\kappa),
	\label{eq:SplittingOfCd}
\end{equation}
where $\mathcal{S}_N^\infty(\kappa)$, resp. $\mathcal{U}_N^\infty(\kappa)$, denotes the span of all 
generalized eigenvectors 
associated with eigenvalues of negative, resp. positive, real part of the matrix $\NA^\infty(\kappa)$; 
Lm.~\ref{lem:consistentsplitting} implies
\begin{equation*}
	\dim\mathcal{S}_N^\infty(\kappa) = d_N^{s}\quad\text{and}\quad \dim\mathcal{U}_N^\infty(\kappa) = d_N^{u}.
\end{equation*}
This property, often referred to as \textit{consistent splitting}, is the key requirement for defining an Evans
function for the truncated problems.

\begin{proof}[Proof of Lm. \ref{lem:consistentsplitting}]
	To begin with, we prove that the spectrum of $\NA^\infty(\kappa)$ does not
	intersect the imaginary axis. To see this, we show 
	that the existence of an imaginary eigenvalue of
	$\NA^\infty(\kappa)$ implies $\kappa\in i\R$.
	Since $\NA^\infty(\kappa)$ is block-diagonal (see \eqref{eq:AinftyIsASum}), it suffices to show this
	separately for each block 
	\begin{equation*}
		\mathcal{A}_{M}(\kappa) =
		\begin{pmatrix}
			\frac{\kappa}{c} & 0 & -\frac{1}{c} & 0\\
			0 & 0 & 1 & 0\\
			0 & 0 & 0 & 1\\
			\delta\lambda\kappa & -\frac{1}{c}\kappa\lambda_M\delta & \delta (\lambda_M-\lambda) &\frac{\kappa}{c}
		\end{pmatrix}
	\end{equation*}
	(recall that $\lambda=\frac{g}{c^2}$, and for $\lambda_M$ see paragraph below \eqref{eq:rhoisexp}). 
	The characteristic polynomial of $\mathcal{A}_M$ is given by 
	\begin{equation}
		\chi_M(\mu;\kappa) := \mu^4 - \frac{2\kappa}{c}\mu^3 + \left( \frac{\kappa^2}{c^2}+\delta(\lambda-\lambda_M) \right) \mu^2
		+2\lambda_M\delta\frac{\kappa}{c} \mu - \frac{\kappa^2}{c^2}\lambda_M \delta,
	\label{eq:CharPol}
	\end{equation}
	and setting $\mu=i\beta$ leads to the equation
	\begin{equation*}
		A\hat\kappa^2 + iB \hat{\kappa} - C = 0
	\end{equation*}
	for $\hat\kappa=-\frac{\kappa}{c}$ with 
	\begin{equation*}
		A= \beta^2 + \lambda_M\delta>0,\; B=2\beta^3+2\lambda_M\delta\beta,\; C= \beta^4 + \delta (\lambda_M-\lambda)\beta^2.
	\end{equation*}
	The calculation 
	\begin{align*}
		-B^2+4AC &= -4\beta^2 (\beta^2+\lambda_M\delta)^2 + 4 (\beta^2+\lambda_M\delta)\cdot\beta^2
		(\beta^2+(\lambda_M-\lambda)\delta)\\
		&= -4\beta^2 (\beta^2+\lambda_M\delta)\delta\lambda \le 0
	\end{align*}
	implies that 
	$\hat{\kappa}\in i\R$, hence $\kappa\in i\R$.

	The first part of the proof implies that the dimension of the stable resp. unstable space is the same for all
	$\kappa$ with $\Re\kappa>0$. For determining their exact dimensions, it thus suffices to consider some special choice
	of $\kappa$, and we choose $\kappa\in\R$, $\kappa>0$ sufficiently large.
	To handle this precisely, we introduce $t:=\mu^{-1}$ and $k:=\left( \frac{\kappa}{c} \right)^{-1}$ 
	and, after multiplying \eqref{eq:CharPol} by $-t^4k^2$, obtain
	\begin{equation}
		\lambda_M \delta t^4 -2\lambda_M\delta k  t^3 - \left( 1 + \delta(\lambda-\lambda_M)k^2 \right) t^2 + 2kt -k^2 = 0.
	\label{eq:CharPol-kt}
	\end{equation}
	Using the Newton polygon method, 
  it is possible to provide approximate expressions for
  the roots $t_j=t_j(k)$, $j\in\{1,2,3,4\}$, of \eqref{eq:CharPol-kt}. 
  The underlying idea is to introduce a rescaling $t = k^\gamma T$ with a suitable
  exponent $\gamma$ which has to be chosen such that the equation resulting from \eqref{eq:CharPol-kt} after
  rescaling and cancelling the highest common power of $k$ retains at least two summands in the leading order.
  A systematic way to find the appropriate exponents is offered by the Newton polygon method (see \cite[Ch.\
  2.8]{ChowHale82}). For \eqref{eq:CharPol-kt}, however, this can also be accomplished directly and yields
	$\gamma\in \{0,1\}$.
	
	In the first case, $\gamma=0$, set $k=0$ in \eqref{eq:CharPol-kt} to obtain the equation
	\begin{equation}
		\lambda_M\delta t^4 - t^2 = 0
		\label{}
	\end{equation}
	which possesses the roots
	\begin{equation}
		t_{1,2}(0) = 0\quad \text{and}\quad t_{3,4}(0) = \pm \frac{1}{\sqrt{\lambda_M\delta}}.
		\label{}
	\end{equation}
	Since the roots $t_{3,4}(0)$ are simple, they persist under perturbations, hence there are two zeros
	$t_{3,4}(k)$ of \eqref{eq:CharPol-kt} with approximate expressions
	\begin{equation}
		 t_{3,4}(k) = \pm \frac{1}{\sqrt{\lambda_M\delta}} + o(1).
		\label{eq:roots-34}
	\end{equation}
	As $t_{1,2}(0)=0$, we cannot infer the sign of $t_{1,2}(k)$.

	In the second case, $\gamma=1$, we plug the ansatz $t=k T$ into \eqref{eq:CharPol-kt}, and cancel the common factor
	$k^2$. Setting $k=0$ yields
	\begin{equation*}
		-(T-1)^2 = 0,
	\end{equation*}
	thus $T_{1,2}(0)=1$ have positive real part; 
	as $T_{1,2}(k)$ are continuous with respect to $k$
	they have positive real part as	well.
	Turning back to the original variables, this means
	\eqref{eq:CharPol-kt} has two roots $t_{1,2}(k)$ close to the origin with approximate expressions
	\begin{equation}
		t_{1,2} = k + o(k).
		\label{eq:roots-12}
	\end{equation}
	Since $k>0$ is supposed to be small, the formulas \eqref{eq:roots-34} and \eqref{eq:roots-12} imply that there are
	three roots ($t_{1,2,3}$) with positive real part and one root ($t_4$) with negative real part, as it was claimed.
\end{proof}
Next, we investigate the zeros of $\chi_M(\mu;\kappa)$ for $\kappa\in i\R$. For this purpose, we introduce $\kappa=iK$
and $\mu=iB$, with $K\in\R$, and consider in the following the real polynomial $p_K(B) := \chi_M(iB;iK)$,
	\[p_K(B) = B^4 + 2KB^3 + (K^2+\delta(\lambda_M-\lambda))B^2 + 2\lambda_M\delta KB + \delta\lambda_M
	K^2.\]
\begin{lem}
	The polynomial $p_K(B)$
	has precisely two real roots for any $K\in\R$, which are distinct for $K\not=0$.
	\label{lem:ZerosOnImAxis}
\end{lem}
\begin{proof}
  \textit{Step 1}: The transformation $(B,K)\mapsto (-B,-K)$ leaves the polynomial invariant. Thus, it suffices to consider
	$K\geq 0$. For $K=0$ the polynomial $p_0(B) = B^4+ \delta(\lambda_M-\lambda)B^2$ has precisely two real roots, namely
	$B_{1,2} = 0$ (since $\lambda_M-\lambda>\lambda_M-\lambda_0\ge 0$). 
	Thus, it suffices to treat the case $K>0$ in the rest of the proof.

  \textit{Step 2}: We show that the polynomial has two roots.
	Since all the coefficients of $p_K(B)$ are positive, any root is negative. Let us consider the signs of
	$p_K(B)$ and its derivatives at special values of $B$ in order to apply the Fourier-Budan theorem (see
	\cite{Barbeau89} for details).
	\begin{table}[bt]
		\centering
		\begin{tabular}{c|ccccc|c}
			$B$ & $p_K(B)$ & $p'_K(B)$ &$p''_K(B)$ &$p'''_K(B)$ &$p''''_K(B)$ & no. of changes\\
			\hline
			0	& + & + & + & + & + & 0\\
			$K/2+\varepsilon$	& + & + & $\pm$ & $-$ & + & 2\\
			$K$	& $-$ & + & + & $-$ & + & 3\\
			$B^\ast$	& + & $-$ & + & $-$ & + & 4\\
		\end{tabular}
		\vspace{0.5cm}
		\caption{Number of sign changes at some distinguished points (with any sufficiently small $\varepsilon>0$ and some
      sufficiently large \mbox{$B^\ast>0$})}
		\label{tab:signchanges}
	\end{table}
	According to Table~\ref{tab:signchanges}, the difference in the number of sign changes between the points $B^\ast$ and $K$ and
	between the points $K$ and $K/2+\varepsilon$ is one in either case. This implies that each of the intervals $(K,B^\ast)$
	and $(K/2+\varepsilon,K)$ contains precisely	one simple zero	for any $K>0$. Since $\varepsilon>0$ was arbitrary, we can even
	conclude that $(K/2,K)$ contains precisely one simple zero.

  \textit{Step 3}: In this final step, we exclude further roots of $p_K(B)$. For sufficiently small $K>0$, the polynomial has no
	inflection point (since $p''_K(B)$ does not have a real zero), hence $p_K(B)$ has at most two real roots, and we are
	done in this case. 
	This means we find some $K_0>0$ such that $p_{K_0}(B)>0$ for all $B\in [0,K_0/2]$; $K_0$ will be used later.
	For arbitrary $K>0$, we show that $p_K(B)>0$ for all $B\in [0,K/2]$.
	To this end, we consider the derivative of $p_K(B)$, now viewed as a function of the two variables $B$ and $K$,
	with respect to $K$ along rays $B=\gamma K$ with $0<\gamma<\frac{1}{2}$. We find
	\begin{align*}
		\frac{\d}{\d K}p_K(\gamma K) &= \gamma\frac{\partial}{\partial B}p_K(\gamma K) + \frac{\partial}{\partial K}p_K(\gamma K)\\
		&= \gamma\left[4(\gamma K)^3-6K(\gamma K)^2 + 2 (K^2+\delta(\lambda_M-\lambda))\gamma K - 2\lambda_M\delta K
		\right]\\
		&\quad +\left[2( (\gamma K)^2+\delta\lambda_M)(1-\gamma)K   \right]\\
		&= 4\gamma^2(\gamma-1)^2K^3 + 2\delta (\gamma(\lambda_M-\lambda) + (1-2\gamma)\lambda_M)K,
	\end{align*}
	hence 
	\begin{equation*}
		\frac{\d}{\d K}p_K(\gamma K) > 0
	\end{equation*}
	for $K>0$ and $\gamma\in(0,1/2]$.
	Consequently, we can state for any fixed $K>0$ and any $B$ with $0<\frac{B}{K}\le \frac{1}{2}$ that
	\begin{equation*}
		p_K(B) = p_K\left(\frac{B}{K}K\right) > p_{K_0}\left(\frac{B}{K}K_0\right)
		>0.
	\end{equation*}
	Lastly, as $p_K(0) = \delta\lambda_MK^2>0$, we actually find  $p_K(B)>0$ for all $B\in [0,K/2]$.
	This means $p_K$ has no zero in the interval $(0,K/2)$. Together
	with the result from Step 2, we have thus shown that $p_K$ possesses exactly two real roots, which are simple in the
	case of	$K\not=0$.
\end{proof}

In the following lemma,
we introduce a suitable notation for the eigenvalues of $\NA^\infty(\kappa)$ and 
merge the results of Lm.~1 and Lm.~2 to obtain a statement on their behaviour for all $\kappa$ in some open domain
containing $\Cpluscl$.
\begin{lem}
	\label{lem:eigenvalues}
	For any $c>c_0$ and any fixed $N\in\N$ the following holds with some open domain $\Omega = \Omega(N,c) \supset \Cpluscl$:
	The $4N+4$ eigenvalues of $\NA^\infty(\kappa)$ can be sorted as continuous functions
	\begin{align*}
		\mu^{s,u_1,u_2,u_3}_n:\Omega\to \C, \quad \kappa \mapsto \mu^{s,u_1,u_2,u_3}_n(\kappa)\quad \text{for any $n\in\{0,\dots,N\}$,}
	\end{align*}
	such that the relations
	\begin{equation*}
		\Re\mu^s_n(\kappa) < 0 \quad \text{and}\quad \Re\mu^{u_3}_n(\kappa) > 0
	\end{equation*}
	 and
	\begin{equation*}
		\sign \Re\mu^{u_1,u_2}_n(\kappa) = \sign \Re\kappa
	\end{equation*}
	are true for any $\kappa\in\Omega$.
\end{lem}

\begin{proof}
	The notation of the eigenvalues has been chosen in such a way that Lm. \ref{lem:consistentsplitting} implies the assertion
	for all $\kappa\in\Cplus$.

	For any $\kappa\in i\R$, Lm. \ref{lem:ZerosOnImAxis} implies that
	these eigenvalues satisfy 
	\begin{equation*}
		\Re\mu^s_n(\kappa)<0,\quad \Re\mu^{u_3}_n(\kappa)>0 \quad\text{and}\quad \Re\mu^{u_1,u_2}_n(\kappa)=0
	\end{equation*}
	for any $n\in\{0,\dots,N\}$.
	As the map $\kappa\mapsto \mu_n^\ast(\kappa)$ is continuous, we find an open neighbourhood $U(\kappa)$ of $\kappa$ on which
	$\Re\mu^s_n<0$ and $\Re\mu^{u_3}>0$ still hold. The signs of $\Re\mu^{u_1,u_2}$ for $\Re\kappa<0$ follow
	from the observation that the characteristic polynomial is left unchanged under the mapping 
	$(\mu,\kappa)\mapsto (-\mu,-\kappa)$.

	Defining
	\begin{equation*}
		\Omega := \Cplus \cup \bigcup_{\kappa\in i\R} U(\kappa) 
	\end{equation*}
	concludes the proof.
\end{proof}

\begin{proof}[Proof of Lm.~\ref{lem:ExOfUandS}]
	By Lm. \ref{lem:eigenvalues} there exists some open set $\Omega\supset\Cpluscl$ such that for any $\kappa\in\Omega$
	\begin{equation*}
		\max_{n\in\{0,\dots,N\}} \Re\mu^s_n(\kappa) < \min_{\substack{j\in\{1,2,3\},\\n\in\{0,\dots,N\}}}\Re\mu^{u_j}_n(\kappa).
	\end{equation*}
	Consequently, there is a positive spectral gap between the spaces $\mathcal{S}_N^\infty(\kappa)$ and
	$\mathcal{U}_N^\infty(\kappa)$ which are defined as the span
	of all (possibly generalized) eigenvectors associated with the spectral sets
	\begin{equation*}
		\left\{ \mu^s_n(\kappa):n\in\{0,\dots,N\}\right\},
	\end{equation*}
	and
	\begin{equation*}
		\left\{ \mu^{u_j}_n(\kappa):n\in\{0,\dots,N\},j\in\{1,2,3\}\right\},
	\end{equation*}
	respectively.
	
	Roughly speaking, this splitting is transported to $\xi=0$ by the flow of \eqref{eq:EN}.
	In fact, hypothesis (A2) on the exponential decay ascertains that the construction of stable and unstable spaces
	due to Alexander, Gardner, Jones (see	\cite[Sect.\ 3.B]{AlexanderGardnerEtAl90}) also applies here to yield the spaces
	$\mathcal{S}_N(\kappa)$ and $\mathcal{U}_N(\kappa)$ 
  (these are $\Phi_{\pm}(\lambda,\tau_0)$ with $\lambda=\kappa$ and $\tau_0=0$ in the notation of
	\cite{AlexanderGardnerEtAl90}), which are complex-analytic with respect to $\kappa$ and which
  are unique as complex-analytic continuations of the uniquely defined restrictions to $\Cplus$.%
	\footnote{The construction of stable and unstable spaces can also be found in \cite{Sandstede02,GardnerZumbrun98}.}
\end{proof}
\begin{rem}
	We note that, since here a spectral gap between the spaces $\mathcal{S}_N(\kappa)$ and $\mathcal{U}_N(\kappa)$ is
	maintained as $\kappa$ crosses the imaginary axis, our situation is much simpler than that of the gap lemma (cf.\
	\cite{GardnerZumbrun98}).
\end{rem}

\begin{proof}[Proof of Thm.~\ref{thm:existenceofevansfunctions}]
	In order to define an Evans function, we choose 
	analytic bases (the existence of which is guaranteed by construction due to Kato, see \cite[Ch.\ II.\ §4.2]{Kato76})
\begin{equation*}
\{\zeta_1(\kappa),\dots,\zeta_{d_N^{s}}(\kappa)\} \text{ and } 
\{\eta_1(\kappa),\dots,\eta_{d_N^{u}}(\kappa)\}
\end{equation*}
of $\mathcal{S}_N(\kappa)$ and $\mathcal{U}_N(\kappa)$, respectively. 
Now, we define the Evans function as 
\begin{equation*}
	D_N(\kappa) := \det (\zeta_1(\kappa),\dots,\zeta_{d_N^{s}}(\kappa),\eta_1(\kappa),\dots,\eta_{d_N^{u}}(\kappa)).
\end{equation*}
As the chosen basis vectors are analytic with respect to $\kappa$, the mapping $D_N:\Omega\to\C$ is analytic as well.

Finally, we prove property \eqref{eq:DoubleZeroOfD} in three steps. 
In the rest of the proof, we write $\mathcal{A}$ instead of $\NA$ and we make $c$ explicit to stress the dependence of
$\NA$ on $c$.\\
\textit{Step 1}: According to \cite[Sect.\ 3.3, and Thm.\ 4.1]{Sandstede02} the statement in \eqref{eq:DoubleZeroOfD} is equivalent to the existence of solutions $v_1,v_2$
	satisfying
	\begin{align*}
		v_{1,\xi}(\xi) &= \mathcal{A}(\xi;0,c) v_1(\xi),\\
		v_{2,\xi}(\xi) &= \mathcal{A}(\xi;0,c) v_2(\xi) + \mathcal{A}^{(1)}(\xi,c) v_1(\xi),
	\end{align*}
	where $\mathcal{A}^{(1)}(\xi,c)$ is uniquely defined by 
	$\mathcal{A}(\xi;\kappa,c) = \mathcal{A}(\xi;0,c) + \kappa\mathcal{A}^{(1)}(\xi,c)$.
	It is this statement we are going to prove.\\
  \textit{Step 2}: The profile $U^c(\xi,y)$ is a stationary solution of the Euler equations in co-moving coordinates and thus 
	solves the system 
	\begin{align*}
		(u^c-c)\rho^c_\xi + v^c\rho^c_y &= 0,\\
		\rho^c \left( (u^c-c) u^c_\xi+ v^c u^c_y \right) &= -p^c_\xi,\\
		\rho^c \left( (u^c-c) v^c_\xi+ v^c v^c_y \right) &= -p^c_y-g\rho^c,\\
		u^c_\xi + v^c_y &= 0.
	\end{align*}
	Deriving by $\xi$ and by $c$, respectively, yields that
	\begin{equation*}
		 \tilde{V}_1 = \partial_\xi U^c\quad \text{and}\quad \tilde{V}_2 = \partial_c U^c
	\end{equation*}
	satisfy
	\begin{equation*}
		(\ast)\;\mathcal{L}^c\tilde{V}_1 = 0\quad \text{and}\quad (\ast\ast)\; \mathcal{L}^c\tilde{V}_2 = \tilde{V}_1,
	\end{equation*}
	i.\ e. $\tilde{V}_1$ is an eigenvector and $\tilde{V}_2$ is a generalized eigenvector associated with the eigenvalue $\kappa=0$ of
	$\mathcal{L}^c$, which is the linearized Euler operator (i.e., essentially the right hand side of
	\eqref{eq:evp-euler-full}).\\
  \textit{Step 3}: By Thm. \ref{thm:evp-spatdyn} relation $(\ast)$ implies the existence of some $V_1$ with
	\begin{equation*}
		V_{1,\xi} = \A(\xi;0,c)V_1(\xi).
	\end{equation*}
	A slight modification of the proof of Thm. \ref{thm:evp-spatdyn} shows that relation $(\ast\ast)$ implies the
	existence of some $V_2$ with
	\begin{equation*}
		V_{2,\xi} = \A(\xi;0,c)V_2(\xi) + \A^{(1)}(\xi,c)V_1(\xi).
	\end{equation*}
	The truncation procedure performed in the second part of Sect. \ref{sec:PfThm1} yields functions $v_1,v_2:\R\to\W_N$,
	for any $N\in\N$, with
	\begin{align*}
		v_{1,\xi}(\xi) &= \mathcal{A}(\xi;0,c) v_1(\xi),\\
		v_{2,\xi}(\xi) &= \mathcal{A}(\xi;0,c) v_2(\xi) + \mathcal{A}^{(1)}(\xi,c) v_1(\xi).
	\end{align*}
	We have thus shown that the assertion in Step 1 is true, and this concludes the proof of property
	\eqref{eq:DoubleZeroOfD}. Hence, the proof of Thm. \ref{thm:existenceofevansfunctions} is complete.
\end{proof}
\begin{rem}
  \textup{(i)} 
  As there is more than one choice of bases, the Evans function is unique only up to a non-vanishing factor. This
non-uniqueness causes no trouble since it does not affect the location of the zeros.
\textup{(ii)}
Property \eqref{eq:DoubleZeroOfD} is a consequence of translational invariance and the presence of a continuum of
	travelling waves parametrized by the speed. This is a well-known property of Evans functions associated with a
	solitary wave, e.g.\ cf.~\cite[p.\ 72ff.]{PegoWeinstein92} for the corresponding statement for solitons in
	the generalized Korteweg-deVries equation and other dispersive equations.
\end{rem}

\section{Low-frequency stability of small-amplitude ISWs}\label{sec:smallwaves}
\subsection{Small-amplitude expressions and stability result}
In this section, we apply the Evans function framework established in Sec.~\ref{sec:TruncatedEvansFunctions} to small-amplitude waves that are 
approximated by Korteweg-deVries solitons.
Based on a concrete description of small waves, we derive explicit expressions for the entries of $\mathcal{A}(\xi;\kappa)$.
This section's goal is to preclude unstable modes in a neighbourhood of the origin;
the precise statement is as follows.
\begin{thm}
  \label{thm:SmallAmplKdVRegime}
	For all $N\in\N$ and any $R_0>0$ there exists some $\varepsilon_0>0$ such that for all $\varepsilon\in(0,\varepsilon_0)$
  the Evans function $D_{N,\varepsilon}(\kappa)$ associated with an ISW of amplitude $\varepsilon^2$ satisfies
	\begin{equation*}
		D_{N,\varepsilon}(0) =0,\quad 	D_{N,\varepsilon}'(0) = 0 
	\end{equation*}
	and
	\begin{equation*}
		D_{N,\varepsilon}(\kappa) \not= 0 \text{ for all $\kappa\in\Cpluscl\setminus\{0\}$ with $\abs{\kappa} < R_0\varepsilon^{3}$.}
	\end{equation*}
\end{thm}

Our proof relies on the slow-fast structure of the problem (E$_N$), for any $N\in\N$, and an Evans function treatment of
the reduced system, in which we recover the eigenvalue problem associated with a KdV soliton!

The starting point is the following remarkable formula for the stream function $\psi^c(\xi,y)$ of a small-amplitude ISW
of speed $c=c_0+\varepsilon^2$ and amplitude roughly $\varepsilon^2$:
\begin{equation}
  \label{eq:smallpsi}
  \psi^c(\xi,y) = a_\varepsilon(\xi) \varphi_0(y) + \O(\abs{a_\varepsilon}^2).
\end{equation}
By virtue of this formula, small ISWs are, to leading order, a product with separated variables; more precisely, 
the height-independent part $a_\varepsilon(\xi)$ describes the horizontal propagation while $\varphi_0(y)$ is a
height-dependent amplification factor.
The function $\varphi_0$, as before, denotes the principal eigenfunction of $\frac{1}{\bar{\rho}'}\partial_y(\bar{\rho}\partial_y)$ 
from Sec.~2 (cf.~p.~\pageref{def:phi0}). 
The function $a_\varepsilon(\xi)$ is a symmetric soliton solution of the equation
\begin{equation*}
  a_{\varepsilon,\xi\xi} = -\frac{\varepsilon^2}{s} a_\varepsilon - \frac{r}{s} a_\varepsilon^2 +  O(\varepsilon^4),
\end{equation*}
which involves the $\bar{\rho}$-dependent coefficients $r$ and $s$ defined by
\begin{equation*}
  s = -\frac{c_0}{2} \frac{\int_0^1 \bar{\rho}\varphi_{0}^2 \d{y}}{\int_0^1 \bar{\rho}(\varphi_0')^2 \d{y}}<0
  \quad\text{and}\quad 
  r = -\frac{3}{4} \frac{\int_0^1 \bar{\rho}(\varphi_0')^3 \d{y}}{\int_0^1 \bar{\rho}(\varphi_0')^2 \d{y}}.
\end{equation*}
For our exponential stratification $\bar{\rho}(y)=\e^{-\delta y}$, these coefficients are (cf.~\cite{Benney66,James97}):
\begin{equation*}
	s = -\frac{c_0}{2\delta\lambda_0}<0
  \quad \text{and} \quad
  r = -\frac{3\delta \pi^3 (\e^{\delta/2}+1)}{2 \left( \frac{1}{4}\delta^2+\pi^2  \right)%
  \left(\frac{1}{4}\delta^2+9\pi^2 \right)}<0.
\end{equation*}
Finally, the important simple relationship
\begin{equation*}
  a_\varepsilon(\xi) = \varepsilon^2 A_\ast(\varepsilon\xi) + O(\varepsilon^4)
\end{equation*}
expresses $a_\varepsilon(\xi)$ in terms of the KdV soliton
\begin{equation}
  A_\ast(\Xi) = -\frac{3}{2r}\, \sech^2\left(\frac{1}{\sqrt{-s}}\, \Xi \right)
  \label{eq:kdvsoliton}
\end{equation}
satisfying
\begin{equation}
  A_{\ast,\Xi\Xi}=-\frac{1}{s}A_\ast -\frac{r}{s}A_\ast^2.
  \label{eq:kdvprofileeq}
\end{equation}

The appearance of the KdV equation in the context of small-amplitude ISWs
permits to extract the essential part of the lengthy expressions for the operators $R_k$ and $S_l$ given in Thm.~\ref{thm:evp-spatdyn}.
We exemplify this procedure for the expression
\begin{equation*}
  {}^0\!\mathcal{A}(\xi,\kappa)_{43} = \scalarprod{\A(\xi,\kappa)U_0^3,U_0^4}.
\end{equation*}
Recalling the notation from Thm.~\ref{thm:evp-spatdyn} and using Eq.~\eqref{eq:smallpsi}, we find
\begin{align}
  \nonumber\scalarprod{\A(\xi,\kappa)U_0^3,U_0^4} &= \int_{0}^{1} (-\bar{\rho}'(y)) S_3(\varphi_0)\varphi_0 \d{y}\\
  \nonumber &= \varepsilon^2\left(\frac{2\lambda_0\delta}{c_0} - \frac{3\delta}{c_0}\int_{0}^{1}\bar{\rho}(\varphi_0')^3
  \d{y} A_\ast(\varepsilon\xi)\right)+ \text{h.o.t.}\\
  \label{eq:A43-leadingorder} &= \varepsilon^2\left(-\frac{1}{s}-\frac{2r}{s} A_\ast(\varepsilon\xi) \right)+ \text{h.o.t.}
\end{align}

Proceeding in this way, we finally end up with the following expressions 
\begin{equation*}
	\mathcal{A}_0 =
	\begin{pmatrix}
		\frac{1}{c_0}\kappa +\O(\kappa\varepsilon^2) & 0 & -\frac{1}{c_0} + \O(\varepsilon^2) & 0\\
		0 & 0 & 1 & 0\\
		0 & 0 & 0 & 1\\
		\lambda_0\delta\kappa + \O(\kappa\varepsilon^2)
		& -\frac{\lambda_0\delta}{c_0}\kappa +\O(\kappa\varepsilon^2) 
		& \frac{2\lambda_0\delta}{c_0}\varepsilon^2 +\O(\varepsilon^4)
		& \frac{1}{c_0}\kappa + \O(\kappa\varepsilon^2)
	\end{pmatrix},
\end{equation*}

\begin{equation*}
	\mathcal{A}_n =
	\begin{pmatrix}
		\frac{1}{c_0}\kappa +\O(\kappa\varepsilon^2) & 0 & -\frac{1}{c_0} + \O(\varepsilon^2) & 0\\
		0 & 0 & 1 & 0\\
		0 & 0 & 0 & 1\\
		\lambda_0\delta\kappa + \O(\kappa\varepsilon^2)
		& -\frac{\lambda_n\delta}{c_0}\kappa +\O(\kappa\varepsilon^2) 
		& \delta(\lambda_n-\lambda_0)+\O(\varepsilon^2)
		& \frac{1}{c_0}\kappa + \O(\kappa\varepsilon^2)
	\end{pmatrix}
\end{equation*}
for $n\geq 1$, and
\begin{equation*}
	\mathcal{B}_{nm} = 
	\begin{pmatrix}
		\varepsilon^3 A_\varepsilon'(\xi) G^{11}_{nm} & \varepsilon^3 A_\varepsilon'(\xi) G^{12}_{nm} & \varepsilon^2 A_\varepsilon(\xi) G^{13}_{nm} & 0\\
		0&0&0&0\\
		0&0&0&0\\
		\varepsilon^3 A_\varepsilon'(\xi) G^{41}_{nm} & \varepsilon^3 A_\varepsilon'(\xi) G^{42}_{nm} & \varepsilon^2 A_\varepsilon(\xi) G^{43}_{nm} 
		& \varepsilon^3 A_\varepsilon'(\xi) G^{44}_{nm}
	\end{pmatrix}
\end{equation*}
for $n,m\in\N$ 
after substituting $a_\varepsilon(\xi)=\varepsilon^2 A_\varepsilon(\xi)$ and neglecting higher order terms
(i.e. $\O(\varepsilon^4+\kappa\varepsilon^4)$).
All of the constants $G^{ij}_{nm}$, which are independent of $\kappa$ and $\varepsilon$, can be computed explicitly; in the
sequel, however, only the following three will be important:
\begin{equation*}
  G^{41}_{00} = \frac{2c_0}{3} \frac{r}{s},\quad 
  G^{42}_{00} = -\frac{4}{3} \frac{r}{s},\quad
  G^{43}_{00} = -\frac{2r}{s}.
\end{equation*}
Note that $G^{43}_{00}$ can be directly read off from Eq.~\eqref{eq:A43-leadingorder}.

In the rest of this section, we shall be occupied with proving that 
the autonomous first-order system
\begin{align}
  \tag{E$_{N,\varepsilon}$}\label{eq:evpneps}
  w'(\xi) &= {}^N\!\mathcal{A}_{\kappa,\varepsilon}[A_\varepsilon,B_\varepsilon] w(\xi) \\
  \label{eq:profileAB}
  \begin{pmatrix}
    A_\varepsilon(\xi)\\B_\varepsilon(\xi)
  \end{pmatrix}'
  &= \begin{pmatrix}
    \varepsilon B_\varepsilon(\xi)\\
    \varepsilon\left( -\frac{1}{s} A_\varepsilon - \frac{r}{s} A_\varepsilon^2  \right) + O(\varepsilon^3)
  \end{pmatrix}
\end{align}
does not have bounded solutions provided that $\abs{\kappa \varepsilon^{-3}}$ be bounded and $\varepsilon$ be small
enough.

We divide the statement of the theorem in two parts corresponding to the following two propositions which jointly prove the
theorem. 
The first proposition reduces the dimension of the linear part of the original problem from $4N+4$ to $2N+4$.
The second proposition states absence of unstable modes in the reduced problem.
\begin{prop}
  \label{prop:exofcmf}
  For any $N$ and for sufficiently small $\varepsilon$ the system \eqref{eq:evpneps} possesses a centre manifold.
\end{prop}
We denote the reduced problem by \eqref{eq:evp-eps}.
\begin{prop}
  \label{prop:redprob}
  For any $R_0$, $N\in\N$ and $\abs{\kappa}<R_0\varepsilon^3$ with $\Re\kappa>0$ the reduced problem \eqref{eq:evp-eps}
  does not possess bounded solutions.
\end{prop}
The two propositions are proved in the subsequent subsections.
In the rest of this section, we arbitrarily fix $N\in\N$ and $R_0>0$.

\subsection{Proof of Propositon~\ref{prop:exofcmf}: Centre manifold reduction}
By introducing $\Lambda:=\kappa\varepsilon^{-3}$, we are in the regime
\begin{equation*}
	0\le \abs{\Lambda}\le R_0.
\end{equation*}
By scaling the dependent variables as 
\begin{align*}
	w_1(\xi) &= {W}_1(\xi),&
	w_{4n+1}(\xi) &= \varepsilon W_{4n+1}(\xi),\\
	w_2(\xi) &= {W}_2(\xi),&
	w_{4n+2}(\xi) &= \varepsilon W_{4n+2}(\xi),\\
	w_3(\xi) &= \varepsilon W_3(\xi),&
	w_{4n+3}(\xi) &= \varepsilon^2 W_{4n+3}(\xi),\\
	w_4(\xi) &= \varepsilon^2 W_4(\xi),&
	w_{4n+4}(\xi) &= \varepsilon^2 W_{4n+4}(\xi),
\end{align*}
for all $n\in\{1,\cdots,N\}$, the problem \eqref{eq:evpneps} takes, to leading order, the form
\begin{align*}
	A_\varepsilon' &= \varepsilon B_\varepsilon,\\
	B_\varepsilon' &= \varepsilon\left(-\frac{1}{s} A_\varepsilon- \frac{r}{s}A_\varepsilon^2
		\right)+O(\varepsilon^3),\\
	W_1' &= O(\varepsilon),\\
	W_2' &= O(\varepsilon),\\
	W_3' &= O(\varepsilon),\\
	W_4' &= O(\varepsilon),\\
	W_{4n+1}' &= O(\varepsilon),\\
	W_{4n+2}' &= O(\varepsilon),\\
	W_{4n+3}' &= W_{4n+4},\\
	W_{4n+4}' &= 	\delta(\lambda_n-\lambda_0)W_{4n+3} + O(\varepsilon^2),
\end{align*}
for all $n$ with $1\le n \le N$. We introduce the two complementary index sets
\begin{equation*}
  I_h := \{4n+3, 4n+4:n=1,\dots,N\} \text{ and } 
  I_c := \{1,\dots,4N+4\}\setminus I_h
\end{equation*}
such that $j\in I_c$ iff $W_j'=O(\varepsilon)$.
From the form of the equations, we easily infer that, for $\varepsilon=0$, the set
	\begin{equation*}
    \mathcal{M}_{0} := \{ W_j = 0: j\in I_h\}
	\end{equation*}
is a centre manifold, which is normally hyperbolic since the partial Jacobian	matrix of this system 
with respect to the variables $W_{4n+3}, W_{4n+4}$ for $n=1,\dots,N$ has a block structure of the form 
\begin{equation}
	\begin{pmatrix}
		Y_1 & \ast & \cdots & \ast\\
		0 & Y_2 &\ddots &\vdots \\
		\vdots & \ddots & \ddots & \ast\\
		0 & \cdots & 0 & Y_N\\
	\end{pmatrix}
\end{equation}
where each ``$\ast$'' denotes a submatrix which is irrelevant in the sequel and $Y_n$ is given by
\begin{equation*}
	Y_n = 
\begin{pmatrix}
		0 & 1\\
		\delta(\lambda_n-\lambda_0) & 0
	\end{pmatrix},
\end{equation*}
thus all the eigenvalues, given by $\pm \sqrt{\delta(\lambda_n-\lambda_0)}$ for $n=1,\dots,N$, are real and different from zero.

By virtue of Fenichel's theorem on the persistence of normally hyperbolic invariant manifolds (see
\cite{Fenichel72,Fenichel79,Jones95}), we conclude that 
(a)~an invariant manifold $\mathcal{M}_{\varepsilon}$ exists for all sufficiently small $\varepsilon$, say $0<\varepsilon<\varepsilon_1$,
(b)~$\mathcal{M}_\varepsilon$ is a graph over $\mathcal{M}_{0}$ and 
(c)~$W_j = O(\varepsilon)$ for $j\in I_h$.

In this way, we obtain a reduced system for the variables $\{W_j:j\in I_c\}$.
After changing to the slow scale
\begin{equation*}
	\Xi := \varepsilon\xi, 
\end{equation*}
setting
\begin{equation*}
  A_\varepsilon(\xi) = \tilde{A}_\varepsilon(\Xi),\quad 
  B_\varepsilon(\xi) = \tilde{B}_\varepsilon(\Xi)
\end{equation*}
and
\begin{equation*}
W_{1}(\xi) = \frac{1}{2}\hat{W}_{1}(\Xi) -\frac{1}{c_0}\hat{W}_{2}(\Xi), \quad
W_{2}(\xi) = \frac{c_0}{2}\hat{W}_{1}(\Xi) +\hat{W}_{2}(\Xi)
\end{equation*}
as well as 
\begin{equation*}
  W_j(\xi) = \hat{W}_j(\Xi) \quad \text{for all $j\in I_c\setminus\{1,2\}$},
\end{equation*}
we find that,
omitting the hats, 
the reduced system is of the form 
\begin{equation}
	\label{eq:evp-eps}
  \tag{$\widehat{\textup{E}}_{N,\varepsilon}$}
\begin{aligned}
	\dot{W}_1 &= O(\varepsilon^2),\\
	\dot{W}_2 &= W_3 + O(\varepsilon^2),\\
	\dot W_3 &= W_4,\\
	\dot W_4 &= {\Gamma}_1{W}_1 + {\Gamma}_2 {W}_2 + {\Gamma}_3 {W}_3 + O(\varepsilon),\\
	\dot W_{4n+1} &= O(\varepsilon),\\ 
	\dot W_{4n+2} &= O(\varepsilon),
\end{aligned}
\end{equation}
with 
\begin{align*}
	\Gamma_1(\Xi) &= \frac{1}{2} \tilde{\Gamma}_1 + \frac{c_0}{2} \tilde{\Gamma}_2 = -\frac{1}{2} \dot{A}_\ast(\Xi) \int_{0}^{1} \bar{\rho}\varphi_0'^3 \d{y}
	=-\frac{c_0r}{3s} \dot{A}_\ast(\Xi),\\
	\Gamma_2(\Xi) &= -\frac{1}{c_0} \tilde{\Gamma}_1 + \tilde{\Gamma}_2= \frac{\Lambda}{s} - \frac{2r}{s}\dot{A}_\ast(\Xi),\\
	\Gamma_3(\Xi) &= -\frac{1}{s} - \frac{2r}{s}A_\ast(\Xi).
\end{align*}
In this system, the KdV eigenvalue problem becomes visible as follows:
For $\varepsilon=0$ we obtain the system
\begin{equation}
	\label{eq:evp-eps0}
  \tag{$\widehat{\textup{E}}_{N,0}$}
\begin{aligned}
	\dot{W}_1 &= 0,\\
	\dot{W}_2 &= W_3,\\
	\dot W_3 &= W_4,\\
	\dot W_4 &= {\Gamma}_1{W}_1 + {\Gamma}_2 {W}_2 + {\Gamma}_3 {W}_3,\\
	\dot W_{4n+1} &= 0,\\ 
	\dot W_{4n+2} &= 0;
\end{aligned}
\end{equation}
the set $\{W_1=0\}$ is invariant for this flow and on this set the only non-trivial equations are those for
$W_2,W_3,W_4$:
\begin{equation}
	\begin{pmatrix}
		W_2\\W_3\\W_4
	\end{pmatrix}_\Xi
	= 
	\begin{pmatrix}
		0 & 1 & 0\\
		0 & 0 & 1\\
		\frac{\Lambda}{s} - \frac{2r}{s}\dot{A}_\ast(\Xi) & -\frac{1}{s} - \frac{2r}{s}A_\ast(\Xi) & 0
	\end{pmatrix}
	\begin{pmatrix}
		W_2\\W_3\\W_4
	\end{pmatrix}.
	\tag{E$_\KdV$}
	\label{eq:evp-kdv}
\end{equation}
This is the eigenvalue problem of the KdV equation associated with $A_\ast(\Xi)$!

\subsection{Proof of Prop.~\ref{prop:redprob}: Absence of growing modes in the reduced system~\eqref{eq:evp-eps}} 
We will use an Evans function argument in order to prove the statement. 
Therefore, we will first show that the system~\eqref{eq:evp-eps}
is contained in the class of eigenvalue problems treated by Pego and Weinstein in~\cite{PegoWeinstein92}.
For this purpose, we have to check their hypotheses (H1-H4) on the smoothness of the matrix $\leftup{N}\mathcal{A}$, limits at infinity,
simplicity of the lowest eigenvalue, and integrability of the deviator.
As it is easy to see that (H1), (H2) and (H4) are satisfied, we
concentrate on (H3), which states that there is a domain in $\C$ such that 
the asymptotic matrix has a unique simple eigenvalue of smallest real part.
To be more precise, let $\chi_{\varepsilon}(\mu;\Lambda)$ denote the characteristic polynomial of the asymptotic
matrix associated with system \eqref{eq:evp-eps} for $\varepsilon>0$; we have the following lemma.
\begin{lem}
	\label{lem:ChiEpsHasSimpleEV}
	There exists some $\varepsilon_2>0$ such that $\chi_{\varepsilon}(\mu;\Lambda)$ has a unique simple eigenvalue of
	smallest real part for all $0\le \abs{\Lambda}\le R_0$ and for all $0\le \varepsilon\le \varepsilon_2$.
\end{lem}
\begin{proof}
Let $\chi_{0}(\mu;\Lambda)$ and $\chi_\KdV(\mu;\Lambda)$ denote the characteristic polynomials of the asymptotic
matrices associated with the systems \eqref{eq:evp-eps0} and \eqref{eq:evp-kdv}, respectively.
By inspection of these matrices, one finds
\begin{equation*}
	\chi_{0}(\mu;\Lambda) = (-\mu)^{2N+1} \chi_\KdV(\mu;\Lambda).
\end{equation*}
The polynomial $\chi_\KdV(\mu;\Lambda) = \mu^3+\frac{1}{s}\mu-\frac{1}{s}\Lambda$ has the following property (see \cite[p.~72]{PegoWeinstein92}): 
There exists some $\nu>0$ such that for all
$\Lambda$ with $\Re\Lambda \ge -\nu$ there is a unique simple root $\mu_\KdV(\Lambda)$ with smallest real part (which is
negative).
Consequently, the polynomial $\chi_{0}(\mu;\Lambda)$ has this property as well, and we have
$\mu_0(\Lambda)=\mu_\KdV(\Lambda)$.

In order to show that $\chi_{\varepsilon}(\mu;\Lambda)$ also possesses this property for sufficiently small
$\varepsilon$,
we invoke the implicit function theorem. 
Let us define 
the compact set 
\begin{equation*}
K:=\{\Re\Lambda\ge -\nu\}\cap \{\abs{\Lambda}\le R_0\}.
\end{equation*}
The equation 
\begin{equation}
	0 = \chi(\mu;\varepsilon,\Lambda):=\chi_\varepsilon(\mu;\Lambda)
	\label{eq:chieqnull}
\end{equation}
has the solution $\mu_0(\Lambda)$ for $\varepsilon=0$, i.e.\ $\chi(\mu_0(\Lambda);0,\Lambda)=0$.
Since $\mu_0(\Lambda)$ is a simple root of $\chi_0(\mu;\Lambda)$, we know
\begin{equation*}
	\frac{\partial}{\partial\mu}\chi(\mu_0(\Lambda);0,\Lambda)\not= 0,
\end{equation*}
thus the implicit function theorem implies that Eq.~\eqref{eq:chieqnull} can be solved for $\mu$ in a neighbourhood of
$(\varepsilon,\Lambda) = (0, \Lambda_0)$ for any $\Lambda_0\in K$. 
What is more, for any $\Lambda_0\in K$ there exist some $\tilde\varepsilon_2>0$ and a smooth function 
\begin{equation*}
	\tilde\mu:(-\tilde\varepsilon_2,\tilde\varepsilon_2)\times \{\Lambda:\abs{\Lambda-\Lambda_0}<\tilde\varepsilon_2\}\to\C
\end{equation*}
with $\chi(\tilde\mu(\varepsilon,\Lambda);\varepsilon,\Lambda)=0$. 
In this way, we obtain an open cover of $K$ and its compactness allows to pass to a finite subcover. Hence, we find some
$\varepsilon_2>0$ and some function 
\begin{equation*}
	\mu:[0,\varepsilon_2)\times K\to\C
\end{equation*}
such that $\mu_\varepsilon(\Lambda):=\mu(\varepsilon,\Lambda)$
is the unique simple zero of smallest real part of $\chi_\varepsilon(\mu;\Lambda)$ for
all $0<\varepsilon<\varepsilon_2$ and all $\Lambda\in K$.
\end{proof}

We have thus shown that hypothesis (H3) holds on the domain 
\begin{equation*}
  \Omega := \{\Lambda\in\C:\Re\Lambda > -\nu\} \cap \{\abs{\Lambda}< R_0\}
\end{equation*}
and, hence, we may treat the system~\eqref{eq:evp-eps} by applying the theory due to Pego and Weinstein. Before doing
so, we recapitulate some of their notation.

Recall that for a linear differential equation 
\begin{equation*}
  \frac{\d{y}}{\d{x}} = \mathcal{A}(x)y,
\end{equation*}
where $y(x)$ is a column vector, the adjoint system is given by
\begin{equation*}
  \frac{\d{z}}{\d{x}} = -z\mathcal{A}(x),
\end{equation*}
where $z(x)$ is a row vector. In the following, we denote by 
\textup{(E$^\ast_\KdV$)}, 
$\left(\widehat{\textup{E}}^\ast_{N,0}\right)$ and $\left(\widehat{\textup{E}}^\ast_{N,\varepsilon}\right)$
the adjoint systems of 
\textup{(E$_\KdV$)}, 
$\left(\widehat{\textup{E}}_{N,0}\right)$ and $\left(\widehat{\textup{E}}_{N,\varepsilon}\right)$,
respectively.

Let $Z^+_\KdV(\Lambda)$, $Z^+_{0}(\Lambda)$, $Z^+_{\varepsilon}(\Lambda)$ and 
$Y^-_\KdV(\Lambda)$, $Y^-_{0}(\Lambda)$, $Y^-_{\varepsilon}(\Lambda)$
denote associated right, resp. left, eigenvectors of the asymptotic matrices normalized in such a way that $Y^-\cdot Z^+ = 1$ holds.
Then, we obtain the following lemma which states the existence of special functions spanning the stable space and the
dual of the unstable space, respectively, directly by applying \cite[Prop.~1.2]{PegoWeinstein92} to each of the systems 
\eqref{eq:evp-kdv}, \eqref{eq:evp-eps0} and \eqref{eq:evp-eps}. 
\begin{lem}{\cite[p.~56, Prop.~1.2]{PegoWeinstein92}}
	\label{lem:ExOfZetaAndEta}
  \textup{(i)}~For $0<\varepsilon<\varepsilon_2$ there are differentiable functions 
	\begin{equation*}
	\zeta^+_\KdV(\xi;\Lambda),\quad \zeta^+_{0}(\xi;\Lambda),\quad \zeta^+_{\varepsilon}(\xi;\Lambda), 
	\end{equation*}
	analytic with respect to $\Lambda\in\Omega$, with the following properties:\\
  $\zeta^+_\KdV(\xi;\Lambda)$ solves \textup{(E$_\KdV$)} and satisfies
	 $\e^{\mu_\KdV(\Lambda)\xi}\zeta^+_\KdV(\xi;\Lambda)\to Z^+_\KdV(\Lambda)$ as $\xi\to\infty$,\\
   $\zeta^+_{0}(\xi;\Lambda)$ solves $\left(\widehat{\textup{E}}_{N,0}\right)$
   and satisfies $\e^{\mu_{0}(\Lambda)\xi}\zeta^+_{0}(\xi;\Lambda)\to Z^+_{0}(\Lambda)$ as $\xi\to\infty$,\\
  $\zeta^+_{\varepsilon}(\xi;\Lambda)$ solves $\left(\widehat{\textup{E}}_{N,\varepsilon}\right)$ 
  and satisfies
	$\e^{\mu_{\varepsilon}(\Lambda)\xi}\zeta^+_{\varepsilon}(\xi;\Lambda)\to Z^+_{\varepsilon}(\Lambda)$ \mbox{as
  $\xi\to\infty$.}\\
	These conditions characterize the functions uniquely up to a constant factor.

\textup{(ii)}~For $0<\varepsilon<\varepsilon_2$ there are differentiable functions 
	\begin{equation*}
	\eta^-_\KdV(\xi;\Lambda),\quad \eta^-_{0}(\xi;\Lambda),\quad \eta^-_{\varepsilon}(\xi;\Lambda), 
	\end{equation*}
	analytic with respect to $\Lambda$, with the following properties:\\
  $\eta^-_\KdV(\xi;\Lambda)$ solves \textup{(E$^\ast_\KdV$)} and satisfies
  $\e^{\mu_\KdV(\Lambda)\xi}\eta^-_\KdV(\xi;\Lambda)\to Y^-_\KdV(\Lambda)$ \mbox{as $\xi\to -\infty$,}\\
  $\eta^-_{0}(\xi;\Lambda)$ solves $\left(\widehat{\textup{E}}^\ast_{N,0}\right)$ 
  and satisfies
	$\e^{\mu_{0}(\Lambda)\xi}\eta^-_{0}(\xi;\Lambda)\to Y^-_{0}(\Lambda)$ as $\xi\to -\infty$,\\
  $\eta^-_{\varepsilon}(\xi;\Lambda)$ solves $\left(\widehat{\textup{E}}^\ast_{N,\varepsilon}\right)$ 
  and satisfies
	$\e^{\mu_{\varepsilon}(\Lambda)\xi}\eta^-_{\varepsilon}(\xi;\Lambda)\to Y^-_{\varepsilon}(\Lambda)$ \mbox{as
  $\xi\to -\infty$.}\\
	These conditions characterize the functions uniquely up to a constant factor.
\end{lem}
With these functions at hand, we can define the Evans functions
\begin{align*}
	D_\KdV(\Lambda) &:= \eta^-_\KdV(\xi;\Lambda)\cdot \zeta^+_\KdV(\xi;\Lambda),\\
	\hat{D}_0(\Lambda) &:= \eta^-_0(\xi;\Lambda)\cdot \zeta^+_0(\xi;\Lambda),\\
	\hat{D}_\varepsilon(\Lambda) &:= \eta^-_\varepsilon(\xi;\Lambda)\cdot \zeta^+_\varepsilon(\xi;\Lambda)
\end{align*}
for the systems \eqref{eq:evp-kdv}, \eqref{eq:evp-eps0} and \eqref{eq:evp-eps}
in the vein of Pego and Weinstein.
For later use, we recall their result on $D_\KdV$.
\begin{lem}
	\label{lem:KdVEvansFunction}
	The Evans function $D_\KdV:\Omega_\KdV\to\C$ is analytic on the domain 
  $\Omega_\KdV = \{\Lambda\in\C:\Re\Lambda > -\nu\}$, with some $\nu>0$, 
  and has the following properties:
	\begin{enumerate}
      \item $D_\KdV(0) = D_\KdV'(0) = 0$, $D_\KdV''(0)\not=0$, and
      \item $D_\KdV(\Lambda)\not=0$ for all $\Lambda\not=0$ with $\Re\Lambda\ge 0$.
	\end{enumerate}
\end{lem}
So far, we have gathered all the ingredients necessary to give the proof of Prop.~\ref{prop:redprob}.
\begin{proof}[Proof of Prop.~\ref{prop:redprob}]
	\textit{Step 1:} 
	For $\varepsilon = 0$, we find special solutions $\tilde{\zeta}_0^+$, $\tilde{\eta}_0^-$ to \eqref{eq:evp-eps0} and
	its adjoint system, namely 
	\begin{align*}
		\tilde\zeta_0^+ &= (0, \zeta^+_\KdV, 0, \dots, 0)^\T,\\
		\tilde\eta_0^- &= (\ast, \eta^-_\KdV,0,\dots,0),
	\end{align*}
	(with $\ast$ appropriately chosen)
	exhibiting the correct decay rate $\mu_0=\mu_\KdV$ for $\xi\to\pm\infty$, respectively.
	Therefore, Lm.~\ref{lem:ExOfZetaAndEta} implies 
	that there are complex constants $\gamma_1,\gamma_2\in\C$ such that 
	\begin{equation*}
		\zeta_0^+ = \gamma_1\tilde\zeta_0^+
		\quad \text{and}\quad
		\eta_0^- = \gamma_2\tilde\eta_0^-,
	\end{equation*}
	hence
	\begin{equation*}
		\hat{D}_0(\Lambda) = \eta_0^- \cdot \zeta_0^+ =\gamma_1\gamma_2 \tilde\eta_0^- \cdot \tilde\zeta_0^+ = \gamma \eta_\KdV^-
		\cdot \zeta_\KdV^+ = \gamma D_\KdV(\Lambda)
	\end{equation*}
	with $\gamma:=\gamma_1\gamma_2$ being constant. Thus, Lm.~\ref{lem:KdVEvansFunction} implies that $\hat{D}_0(\Lambda)$ does not vanish in $\Omega$ except
	for $\Lambda=0$ where a double zero is present.\\
	\textit{Step 2:}
	By Lm. \ref{lem:ExOfZetaAndEta} the functions $\zeta^+_\varepsilon,\eta^-_\varepsilon$, hence $\hat{D}_\varepsilon(\Lambda)$, also exist for
	$0\leq \varepsilon < \varepsilon_2$.\\
		The Evans function $\hat{D}_\varepsilon(\Lambda)$ is analytic in $\Lambda$ (and continuous with respect to $\varepsilon$), 
	thus we can compare the numbers of zeros of $\hat{D}_\varepsilon$ and $\hat{D}_0$ inside a given domain by invoking Rouch\'es theorem.
	First, we consider a small open ball $U_0$ centred at the the origin. By choosing $\varepsilon$ sufficiently small, say
	$0\le \varepsilon<\varepsilon_3$ we ensure
	that $\abs{\hat{D}_\varepsilon(\Lambda)-\hat{D}_0(\Lambda)}<\abs{\hat{D}_0(\Lambda)}$ holds on the boundary $\partial U_0$; this is
	possible since $\hat{D}_\varepsilon$ is continuous in $\varepsilon$, coincides with $\hat{D}_0$ for $\varepsilon=0$, and
	$\hat{D}_0(\Lambda)$ does not vanish on $\partial U_0$. 
  Therefore, the number of zeros of $\hat{D}_\varepsilon$ in $U_0$ equals the number of zeros of $\hat{D}_0$ in
  $U_0$, which is two by the previous step.\\ 
  Second, for any open ball $U\subset \overline{\Omega}\setminus U_0$, we similarly find that $\hat{D}_\varepsilon$
	does not vanish on $U$ provided $\varepsilon$ is sufficiently small. Since $\overline\Omega\setminus U_0$ is compact, we may
	pass to a finite subcover to conclude that there exists some $\varepsilon_4>0$ such that $\hat{D}_\varepsilon$ does not vanish on
	$\overline\Omega\setminus U_0$ for all $0\le \varepsilon<\varepsilon_4$.\\
	\textit{Step 3:}
	On the other hand, $\hat{D}_\varepsilon$ has a double zero in	$\Lambda=0$ due to the generalized eigenvectors associated
  with horizontal shifts and changes in speed, see 
  Thm.~\ref{thm:existenceofevansfunctions}. As we have shown that there are at most two zeros, we see that
	$\hat{D}_\varepsilon(\Lambda)\not=0$ for all $\Lambda\in\Omega\setminus\{0\}$ with $\Re\Lambda\ge 0$ and for all
	$0\le\varepsilon<\varepsilon_0:=\min\{\varepsilon_k:k=1,\cdots,4\}$; this concludes the proof.
\end{proof}
 
\section{Perspective}
\label{sec:outlook}

In order to illustrate the wider perspective of our approach, we finally give the following conjecture.
\begin{conj*}
  Consider a regular ISW $U^c(\xi,y)$. Then:\\
	(i)~After suitable normalization of the $N$-th order truncated Evans function $D_N$, the limit
  \[D = \lim_{N\to\infty} D_N\] 
  exists on $\Cpluscl$. This $D$ satisfies
  \begin{equation*}
    D(0) = 0\quad \text{and} \quad D'(0) = 0.
  \end{equation*}

  (ii)~If the amplitude of $U^c$ is sufficiently small, $D$ satisfies, in addition,
  \begin{equation*}
    D''(0) \not= 0,\quad \text{and}\quad D(\kappa)\not=0\text{ for }\kappa\in\Cpluscl\setminus\{0\}.
  \end{equation*}
%
%
\end{conj*}
A natural next step towards proving part~(ii) of this conjecture consists of investigating whether the result on
small-amplitude waves in Sec.~\ref{sec:smallwaves} allows for taking the limit $N\to\infty$. 
More concretely, we ask at first whether the two functions
$\zeta^+(\xi;\varepsilon,N)$ and $\eta^-(\xi;\varepsilon,N)$, which span
the stable space and the dual of the unstable space, respectively, converge in an appropriate sense as $N\to\infty$. This is
obvious for $\varepsilon=0$ because $\zeta^+(\xi;\varepsilon,N+1)$ is obtained from $\zeta^+(\xi;\varepsilon,N)$ simply by
appending a zero (and similarly for $\eta^-$). Would such a result be true for $\varepsilon>0$ as well?

\subsection*{Acknowledgments}
\noindent I would like to express my sincere gratitude to my PhD advisor, Prof.~Heinrich Freist\"uhler (University of
Konstanz), for having proposed the fascinating project on \textit{internal waves} to me and, in particular, for
inspiring discussions and helpful suggestions on the present work.
I also thank the \textit{Studienstiftung des deutschen Volkes} for their support by means of a PhD fellowship. 
\bibliographystyle{abbrv}
\bibliography{./mybib}
\end{document}